\PassOptionsToPackage{table,dvipsnames}{xcolor}
\documentclass[12pt]{article}

\usepackage[english]{babel}
\usepackage[utf8]{inputenc}
\usepackage{geometry}
\usepackage{multicol,float}
\usepackage{multirow}
\usepackage{array}
\usepackage[affil-it]{authblk}

\usepackage{amsthm}
\usepackage{amsmath}
\usepackage{amsfonts}
\usepackage{amssymb}
\usepackage{graphicx}
\usepackage{todonotes}
\usepackage{complexity}
\usepackage{adjustbox}
\usepackage{changepage}
\usepackage[ruled,vlined]{algorithm2e}
\usepackage[table,dvipsnames]{xcolor}
\usepackage{subfig}
\usepackage{cancel}
\usepackage{comment}
\usepackage{ulem}
\usepackage{empheq}
\usepackage{appendix}
\usepackage{natbib}
\setcitestyle{authoryear,open={(},close={)}} 

\usepackage{hyperref}
\hypersetup{
  colorlinks   = true, 
  urlcolor     = cyan, 
  linkcolor    = blue, 
  citecolor   = blue 
}

\SetKwInput{KwData}{Input}
\SetKwInput{KwResult}{Output}

\newtheorem{prop}{Proposition}
\newtheorem{thm}{Theorem}

\makeatletter
\def\blfootnote{\xdef\@thefnmark{}\@footnotetext}
\makeatother

\makeatletter
\def\ps@pprintTitle{%
  \let\@oddhead\@empty
  \let\@evenhead\@empty
  \let\@oddfoot\@empty
  \let\@evenfoot\@oddfoot
}
\makeatother

\title{A bilevel approach for compensation and routing decisions in last-mile delivery\blfootnote{\textit{Email addresses:}\\ \texttt{mcerulli@unisa.it} (Martina Cerulli), \texttt{claudia.archetti@unibs.it} (Claudia Archetti), \texttt{elena.fernandez@uca.es} (Elena Fern\'andez), \texttt{ivana.ljubic@essec.edu} (Ivana Ljubi\'c)}}

\author[1]{Martina Cerulli}
\author[2]{Claudia Archetti}
\author[3]{Elena Fern\'andez}
\author[2]{Ivana Ljubi\'c}

\affil[1]{\small{Department of Computer Science, University of Salerno, 84084 Fisciano, Italy}}
\affil[2]{\small{Department of Information Systems, Decision Sciences and Statistics, ESSEC Business School, 95021 Cergy Pontoise, France}}
\affil[3]{\small{Departament of Statistics and Operational Research, Universidad de C\'adiz, Puerto Real, Spain}}

\date{}

\begin{document}
\maketitle
\vspace*{-1cm}
\hspace*{-6mm}\fcolorbox{red}{white}{\parbox{\textwidth}{This paper has been accepted for publication on \textbf{Transportation Science} (Vol.\ 58, No.\ 5, Pages 1076--1100, 2024) . The \textbf{final published version} is available at \url{https://doi.org/10.1287/trsc.2023.0129}, along with its online appendix.}}
\vspace*{4mm}
\begin{abstract}
In last-mile delivery logistics, peer-to-peer logistic platforms play an important role in connecting senders, customers, and independent carriers to fulfill delivery requests. As the carriers are not under the platform's control, the platform has to anticipate their reactions, while deciding how to allocate the delivery operations. Indeed, carriers' decisions largely affect the platform's revenue. In this paper, we model this problem using bilevel programming. At the upper level, the platform decides how to assign the orders to the carriers; at the lower level, each carrier solves a profitable tour problem to determine which offered requests to accept, based on her own profit maximization. Possibly, the platform can influence carriers' decisions by determining also the compensation paid for each accepted request. The two considered settings result in two different formulations: the bilevel profitable tour problem with fixed compensation margins and with margin decisions, respectively. For each of them, we propose single-level reformulations and alternative formulations where the lower-level routing variables are projected out. A branch-and-cut algorithm is proposed to solve the bilevel models, with a tailored warm-start heuristic used to speed up the solution process. Extensive computational tests are performed to compare the proposed formulations and analyze solution characteristics.
\end{abstract}%

\section{Introduction}

The term ``last-mile delivery'' refers to the final leg in a Business-To-Consumer delivery service whereby the consignment is delivered to the recipient, either at the recipient’s home or at a collection point \citep{archetti2021}.
In today's last-mile delivery systems, a rather common scenario involves a platform that receives orders from customers who require delivery \citep{punel2017, agatz2012,lei2019}. The platform charges customers for delivery, but delivery operations are performed by independent carriers who have spot contracts with the platform, and receive compensation for each delivered order. The platform is responsible for allocating the orders to the carriers, and, possibly, determining the compensation for each order. On the other hand, carriers decide whether to accept the assignments offered by the platform.
This situation contrasts with delivery systems in which there is a single decision maker and applies to real-world routing applications where multiple stakeholders pursue their own objectives, which may be conflicting.

Clearly, the goal of the platform is maximizing its own profit. However, carriers also act on the base of their own profit maximization, and may possibly reject delivery assignments in case the corresponding compensation is deemed insufficient. Thus, the platform has to anticipate carriers’ decisions in order to optimize its own profit. In this paper, we model this sequential decision-making process using bilevel programming \citep{dempe}. 
The platform acts as a leader, getting a profit associated with each delivered order, corresponding to the price paid by the customer minus the compensation given to the carriers. At the lower level, each carrier maximizes the difference between the compensation offered by the platform for the assigned order and the routing cost incurred in the delivery.

Problems of this type are faced in practice by the so-called peer-to-peer transportation platforms, as, among others: \textit{Uber Eats}, which is Uber's food delivery service that connects customers with local restaurants and independent delivery drivers; \textit{Glovo}, which offers a wide range of delivery services, including food, groceries, and courier services; \textit{Amazon Flex}, which allows individuals to deliver Amazon packages using their own vehicles, offering flexibility and earning opportunities.
These are just a few examples, and we address the reader to \cite{ALNAGGAR2021102139} and \cite{horner2021} for a more complete list of existing peer-to-peer delivery platforms.
These platforms dynamically connect service (e.g., a ride, a delivery) requests with independent carriers that are not under the platform’s control. As a result, the platform cannot guarantee that an offered request will be accepted by the service providers (carriers), who base the selection on their net profits. In some cases, the platform may influence carriers' decisions by selecting, not only the requests offered to each of them but also the compensation associated with each request.

Throughout this paper we assume that the carriers deliver all accepted orders in one single route, hence the lower-level problem faced by each driver corresponds to a Profitable Tour Problem (PTP) \citep{feillet2005}. The PTP belongs to the class of Vehicle Routing Problems (VRPs) with profits. In PTP, a vehicle, starting from a central depot, can visit a subset of the available customers, collecting a specific profit whenever a customer is visited. The objective of the problem is the maximization of the net profit, i.e., the total collected profit minus the total route cost.

Two different settings for the studied problem are considered: the Bilevel PTP with Fixed Margins (BPFM) and the Bilevel PTP with Margin Decisions (BPMD).
In the first setting, the compensations paid to the carriers are fixed in advance. At the upper level, a platform offers disjoint subsets of a given set of items (parcels/orders) to the set $K$ of carriers, whereas, at the lower level, each driver $k\in K$ solves a PTP and decides which items she accepts to serve, as well as the route she follows. Both the platform and the carriers aim to maximize their net profits, calculated differently at the two levels: the profit of the platform is the difference between the price paid by the customers and the compensation paid to the carriers; the profit of the carriers is the difference between the compensation received for the delivered items and the routing costs. We solve both a value function reformulation and a no-good cuts based reformulation of this problem through a branch-and-cut approach. The second setting models the case where the leader may influence the decisions of the carriers by determining not only the sets of items to offer to each driver but also the compensation paid for each of them.
We propose two different bilevel formulations for the latter problem and solve a value function reformulation (in two versions) through a branch-and-cut approach. We further discuss the link between these two bilevel formulations, before comparing them computationally.

The contributions of this paper can be summarized as follows:
\begin{itemize}
    \item We show that considering a single-level formulation for the problem leads either to an overestimation or to an underestimation of the platform's profit. The former is obtained assuming that the carriers can only accept the whole offered bundle of deliveries or refuse it. The latter is obtained assuming a profit-sharing between the platform and the set of carriers.
    \item We introduce the bilevel PTP with fixed compensation margins (i.e., BPFM), and with margin decisions (i.e., BPMD), where the platform acts as the leader who assigns customer orders to carriers, who act as followers; in the BPMD, the leader also defines the compensation for each item. With it, we fill the gap existing in the last-mile delivery literature regarding simultaneously considering upper-level compensation and lower-level routing decisions. 
    \item We provide bilevel formulations for both problems, as well as their corresponding single-level reformulations. We start with the BPFM, and its formulation, properties, and single-level reformulations are used to introduce the more complex case, i.e., the BPMD. For this problem, we propose two alternative bilevel formulations, that are compared to prove that they provide equivalent bounds. For readers who are not familiar with bilevel optimization, which is at the core of our contribution, we present a brief introduction to this topic in Online Appendix~A. 
    \item We propose alternative formulations where the lower-level routing variables are either included or projected out. We also introduce some valid inequalities which can be used to strengthen the proposed formulations.
    \item We develop a branch-and-cut algorithm for the solution of the problems where optimality, as well as feasibility cuts, are inserted dynamically. We also propose a warm-start Mixed-Integer Programming (MIP) heuristic to speed up the solution process. The heuristic is based on the solution of the BPFM under additional constraints.
    \item We perform extensive tests on instances adapted from benchmark instances for related problems to compare the different formulations we propose. We also compare the bilevel solutions with the ones obtained through two single-level formulations, modeling different problem settings, with the aim of highlighting the advantage for the platform of considering the carriers as independent agents optimizing their own profit. We further analyze the gain of the platform when considering margin decisions in our setting.
\end{itemize}

The paper is organized as follows. In Section~\ref{sec:literature} we revise the relevant literature. Section~\ref{sec:formulations} introduces formally the problems under study. Then, Section~\ref{sec:fixed_margins} focuses on the case where the margins are fixed, whereas Section~\ref{sec:variable_margins} is devoted to the case where the margins are decision variables, for which we propose two formulations, which we compare in Section~\ref{sub:comparison}. In particular, in Section~\ref{sub:proj} we derive, for the problem with fixed margins, a new formulation by projecting out the routing variables of the follower's problem in the upper-level formulation;
in Section~\ref{sub:tmax}, we discuss how the proposed formulation changes if the setting of the followers' problem is modified by considering a maximum route duration constraint instead of a capacity constraint (measured as the maximum number of packages that each carrier can deliver). In Section~\ref{sec:valid_inequalities}, we introduce some valid inequalities that strengthen the proposed formulations. In Section~\ref{sec:solution} the solution approaches for the proposed formulations are discussed. Section~\ref{sec:results} describes the computational experiments and presents the numerical results. Finally, Section~\ref{sec:conc} concludes the paper.

\section{Literature review}\label{sec:literature}

The study of peer-to-peer logistic platforms has experienced a significant increase in recent years. For a comprehensive overview, we address the readers to \citet{agatz2012,cleophas2019,wang2019,lei2019}.

A wide literature on suppliers' (carriers, drivers) selection in peer-to-peer logistic services either overlooked the behavior of suppliers or assumed that their preferences are known in advance, sometimes together with a carrier's bid on services \citep{kafle2017}. 
Often, suppliers' responses are assumed to be predeterminable: all suppliers will accept matches as long as they are stable \citep{wang2018} or meet some constraints, in the form, for example, of an upper bound on the extra driving time/distance \citep{masoud2017,arslan2019}. More recently, the setting of peer-to-peer transportation platforms offering a menu of packages to occasional drivers, which we consider in the current work, has been addressed in \citet{mofidi2019,horner2021,ausseil2022}. As in our framework, the suppliers are not employed by the platform, thus the platform does not have perfect knowledge of the suppliers' preferences related to which requests they would be willing to accept. 
In \citet{mofidi2019}, the platform decides the composition of multiple, simultaneous, personalized recommendations to the suppliers, who then select from this set. 
It is assumed that the platform is able to estimate the expected value of suppliers' utility for each alternative assignment. A deterministic bilevel optimization model is thus presented, in which the platform takes as input the expectation of suppliers’ estimated utilities to make recommendation decisions. 
\citet{horner2021} propose another bilevel formulation, based on the deterministic formulation presented in \citet{mofidi2019}, but adjusted by considering stochastic selection behaviors. A single-level relaxation is then proposed and a Sample Average Approximation method is used to optimize the expected value of the objective function over a sample of scenarios for the drivers' behavior. 
Also \citet{ausseil2022} consider a multiple scenario approach, repeatedly sampling potential drivers' selections, solving the corresponding two-stage decision problems, and combining the multiple different solutions through a consensus function. Neither routing nor compensation decisions are taken into account in these models, which we instead consider in our paper. The deficiency in customized incentive systems, in particular, could potentially jeopardize the satisfaction of both the requester and the deliverer in terms of their utility and profit, respectively. This is implicitly highlighted in \cite{horner2021}, when stating that the proposed methods achieve good performances as long as the drivers are well compensated, i.e., when a percentage of 80\% of the platform revenues goes to them. In support of this statement, \cite{hong2019} propose a Stackelberg game to model the interaction between the platform, which decides both the delivery fees and the paths, and the drivers, who, based on the distance and their utility, make decisions about whether to participate in the delivery process or not. The computational experiments show that including decisions about the compensation level in the process can significantly improve delivery efficiency, and reduce delivery costs compared to traditional delivery methods. In \cite{gdowska2018}, both a professional delivery fleet and a set of occasional carriers are taken into account. Whereas the professional fleet is owned by the platform, the occasional carriers are independent, and can only deliver one parcel. Each delivery request has a fixed probability of being rejected by the occasional carriers. A compensation mechanism is considered to determine the fee to pay to each occasional driver in case a delivery request is accepted. \cite{barbosa2023} extend the model introduced in \cite{gdowska2018} by implementing a golden-section search method that determines the best compensation to offer for each request, considering the probability of rejection dependent on the compensation. 

Whereas compensation decisions have been optimized in a bilevel setting in some of the studies listed above, none of them explicitly models routing decisions at the lower level, as is done in this paper. In fact, an important feature of our approaches relies on the routing nature of the lower-level problem, and in the following, we review works where bilevel optimization is used to model VRPs, as  \cite{du2017,nikolakopoulos2015,parvasi2019,marinakis2008,ning2017}. All these works propose metaheuristics to solve the considered problems, and, more specifically, genetic algorithms. In \cite{du2017}, a multi-depot VRP is considered, and at the upper level, the assignment of customers to depots is decided, whereas depots-customers routing decisions are taken at the lower level. \cite{nikolakopoulos2015} address the VRP with backhauls and Time Windows, where a backhaul is a return trip to the depot, during which the vehicles pick up loads from the visited customers. The goal of the leader is to minimize the number of vehicles involved, whereas the follower aims to minimize the duration of the routes. A bilevel bi-objective formulation is proposed in \cite{parvasi2019} to model the VRP where the involved vehicles are school buses. At the upper level, a transportation company selects some locations from a set of potential bus stop locations (first objective) and determines the optimal bus routes among the selected stops (second objective). 
At the lower level, students are allocated to a stop or to another transportation company in order to minimize the time spent on buses. A bilevel location VRP is studied by \cite{marinakis2008}. The upper level concerns decisions at the strategic level, i.e., the optimal locations of facilities. The lower level is about operational decisions regarding optimal vehicle routes. In \cite{ning2017}, a bilevel model is used to formulate the VRP with uncertain travel times. The leader aims at minimizing the total waiting times of the customers, and the followers want to minimize the waiting times of the vehicles before the beginning of customers’ time windows. The uncertain bilevel model is reformulated into an equivalent deterministic one.

\cite{calvete2011} consider a multi-depot VRP within a production–distribution planning problem. At the upper level, a distribution company orders from a manufacturing company the items that must be supplied to the retailers, whereas deciding on the allocation of these retailers (who play the role of customers) to each depot and on the routes that serve them. The manufacturing company, which is the follower, decides what manufacturing plants will produce the ordered items. Both players want to minimize their own costs. An ant colony optimization approach is developed to solve the bilevel model. The same conflicting agents are considered in \cite{camachovallejo2021}, but with different objectives: the distribution company aims at maximizing the profit gained from the distribution process and minimizing CO2 emissions; the manufacturer aims at minimizing its total costs. The upper level is thus a bi-objective problem. A tabu search heuristic is designed to obtain non-dominated feasible solutions for the distribution company. A hybrid algorithm combining ant colony optimization and tabu search is proposed in \cite{wang2021} to solve a bilevel problem modeling the location-routing problem with cargo splitting, under four different low-carbon policies. The leader is the engineering construction department, which decides on the distribution center location. The follower takes the distribution department as the decision-maker to solve the VRP.

In \cite{santos2020}, a bilevel formulation is proposed to model the VRP with backhauls, without taking into account time windows considered in \citet{nikolakopoulos2015}.
At the upper level, a shipper aims at minimizing the transportation costs by integrating delivery and pickup operations in the routes, whereas at the lower level a set of carriers, acting together, want to maximize their total net profit. The carriers, who may also serve requests from other shippers, may not be willing to collaborate with the shipper. To motivate the carriers to perform integrated routes, the shipper pays them an additional incentive. 
A reformulation is used to build an equivalent single-level problem. 

In some cases, a bilevel approach has been proposed to address a routing problem even if the decision maker of both levels is the same. For instance, \cite{pardalos2007} formulate the capacitated VRP as a bilevel problem, where the first-level decisions concern the assignment of customers to the routes, and the second-level decisions determine the actual routes. In \cite{handoko2015}, the last-mile delivery problem faced in an urban consolidation center, which corresponds to a PTP with multiple vehicles, is modeled using bilevel programming.  At the upper level, the customers to serve are selected, in order to maximize the profit of the carriers' alliance. At the lower level, a Capacitated VRP deals with the optimization of the route given the set of selected customers.

None of the works mentioned above deals with BPFM or BPMD so we now introduce a formal description of these problems.

\section{Problems definition and formulations}\label{sec:formulations}

In this section, we provide a formal description of the problems we address, whose corresponding mathematical formulations are proposed in the next sections (Section~\ref{sec:fixed_margins} and \ref{sec:variable_margins}). After introducing the notations we will use throughout the paper, as well as the two problems we address (BPFM and BPMD), we show in Section~\ref{sub:singlelev} that using single-level formulations leads to a misprediction of the platform's profit. 

\textit{Input sets and parameters\\}
Sets and parameters used in the definition of the problems are listed below.
\begin{itemize}
	\item $\mathcal{G}=(V, A)$: routing network over all nodes $V_0 = \{0,\dots,n\}$; in particular, $V = V_0 \setminus \{0\}$ corresponds to the set of customers to serve, and node $0$ to the depot where routes start and end;
	\item $K$: index set of carriers/drivers/followers/vehicles; 
	\item $p_i$: price that is paid to the platform if customer $i$ is served; 
	\item $\bar p^k_i$: compensation paid by the platform to carrier $k$ if she accepts to serve customer~$i$;
	\item $c^k_{ij}$: arc $\left(i,j\right)$ weight representing travel time for carrier $k$; we assume that arc weights $c^k_{ij}$ satisfy the triangle inequality;
	\item $b^k$: upper bound on the number of items carrier $k$ can deliver ($b^k < n$ for all $k$); 
	\item $t_{max}^k$: upper bound on the duration of the route for carrier $k$. 
\end{itemize}
\vspace*{2mm}

\textit{Notations\\}
In the rest of this paper, we denote by $\delta^+(i)$ ($\delta^-(i)$) the set of arcs exiting from (entering) vertex $i$, and by $\mathcal{T}$ be the set of all routes in graph $\mathcal{G}=(V, A)$. Moreover, we denote by $T \in \mathcal{T}$ an arbitrary route, with $V(T)$ and $A(T)$ being the set of vertices and arcs visited and traversed by the route, respectively. Furthermore, let $C^k(T) = \sum\limits_{(i,j) \in A(T)} c^k_{ij}$ be the arc cost of tour $T$ for driver $k$.

\vspace*{2mm}
\textit{The BPFM and the BPMD\\}
We consider a single-leader multiple-follower Stackelberg game in which there is a set of items $I$ that needs to be delivered to a corresponding set of customers $V$. Each customer $i \in V$ requires exactly one item in $i \in I$ (multiple items required by the same customer are considered as multiple duplicated customers). For this reason, in the following, we will refer just to the set $V$ (and $V_0$ when including the depot), either when referring to the delivered items, or to the served customers. 
An intermediary platform, acting as a leader, receives a price $p_i$ for each item to be delivered. Given a set $K$ of potential carriers (e.g., occasional drivers), the intermediary searches for carriers that can deliver these items, and pays to carrier $k \in K$ a compensation $\bar p^k_i$, $0 < \bar p^k_i < p_i$, for each delivered item, i.e., for each served customer. The difference between the price $p_i$ and the compensation $\bar p_i^k$ is the net profit of the platform in case item $i$ is delivered by carrier $k$. Note that it is 0 in case the item is not delivered by any carrier. 
The net profit associated with item $i$ expressed as a fraction of the price is defined as the ``profit margin'', i.e., $\frac{p_i-p^k_i}{p_i}$.
The leader has to create $|K|$ disjoint subsets of items, each of them to be offered to a carrier. We call $P_k$ the subset of items offered to carrier $k \in K$. 
Each carrier $k \in K$ receives the proposal, and, based on her net profit, decides on a subset of customers $Q_k \subseteq P_k$ to accept to serve. To this end, each carrier solves a PTP with respect to the given set of items (customers) $P_k$. A carrier can refuse to deliver some items, in which case the intermediary's margin for this item becomes zero.
The goal of the leader is to make a call to the carriers, so as to maximize its revenue, which is defined as
$$\sum_{k \in K} \sum_{i \in Q_k} (p_i - \bar p^k_i).$$

We consider two alternative assumptions concerning the number of packages to be served: either we assume that carrier $k$ cannot deliver more than $b^k$ items, or we assume that there is a travel time limit for the follower, $t^k_{max}$.

We assume, without loss of generality, an \textit{optimistic} bilevel setting, i.e., for a given leader's choice, if follower $k$ has multiple optimal responses determined by different sets $\tilde{Q}_k$ of items to be delivered, she will accept to deliver the items that are more favorable to the leader. This means that the follower $k$ will choose to deliver the subset
\begin{equation*}
     Q^*_k \in \arg \max\limits_{\tilde{Q}_k} \{ \sum_{i \in \tilde{Q}_k} (p_i - \bar p^k_i)~:~\tilde{Q}_k \text{ is optimal for the follower $k$}\},
\end{equation*}
where $\tilde{Q}_k$ is optimal for the follower $k$ if it is an optimal solution of the PTP solved by the follower $k$ with respect to the set of items $P_k$ assigned to her by the leader, i.e.:
\begin{align*}
    \tilde{Q}_k \in \arg\max_{Q_k} \{ \sum_{i\in Q_k} \bar p_i^k -{\sum_{\substack{(i,j)\in A:\\ i,j\in Q_k}}} c^k_{ij}
   ~:~ Q_k \subseteq P_k, \; \text{and the route serving nodes $i \in Q_k$ is feasible}\}.
\end{align*}
This is without loss of generality because, in the implementation, when we find the optimistic solution, the leader can a-posteriori add a small $\epsilon$ to the compensation value for the proposed parcels to break the ties.
We further assume that there is no communication between the carriers, i.e., a carrier $k$ is not aware of what is offered to carriers $k' \in K \setminus\{k\}$. Thus, there is no possible bargaining, and no need to establish a generalized Nash equilibrium between the multiple followers' solutions.
Nevertheless, the problems cannot be seen as single-level, because the carriers are selfish agents and do not have to collaborate with the intermediary. Their major goal is to maximize their own net profits, and, therefore, a solution that is optimal for the leader is not necessarily optimal for the follower. We provide in the following section two different single-level formulations, as well as an illustrative example to support this claim.

In the BPFM we assume that the leader does not decide on the compensations to pay to the carriers, i.e., $\bar p_i^k$ is fixed and given a-priori. In the BPMD, instead, we assume that the intermediary platform decides, in addition to the assignment of customers to carriers, the compensation to pay to each carrier (measured as the fractional margin of the price obtained by the platform). As the compensation is a fraction of the price, what remains is the ``margin'' gained by the platform, so we use the term ``margin optimization''. In particular, we consider $|M|$ different margin values that the intermediary platform can choose for each item. This set is the same for all items. We denote as $p_{mi}$ the profit gained by the platform when applying margin $m\in M$ to item $i$.

\subsection{Single-level formulations of related problems}\label{sub:singlelev}
We discuss here two different single-level formulations of the problem faced by the platform, which lead to an upper and a lower bound of the platform's profit, respectively.

Let us suppose that the platform maximizes its own profit while imposing that the profit of the carriers is larger than a certain threshold, e.g., their willingness to accept (WTA), and assuming that $P_k = Q_k$ once this threshold is satisfied. This last assumption defines a different problem setting in which carriers are not allowed to decide on the subset of the assigned items to serve, but can only accept or reject the whole assigned bundle. Assume also that the WTA value for all carriers is 0, i.e., they are willing to accept to serve in case the net profit is non-negative. This single-level setting, which we call the WTA-PTP, involves two types of binary variables: $\alpha^k_i$ for each $i\in V_0$ and $k\in K$, which is 1 if item $i$ is offered to carrier $k$; and $z^k_{ij}$ for all $(i,j)\in A$ and $k\in K$, which is $1$ if arc $\left(i,j\right) \in A$ is traversed by carrier~$k$. The WTA-PTP can be formulated as:
\begin{subequations}\label{eq:wta-ptp}
\begin{empheq}[left=(\mathsf{\textcolor{OliveGreen}{WTA\text{-}PTP}})\empheqlbrace]{align}
\max\limits_{{\alpha},{z}} &  \; \sum_{k \in K}\sum_{i \in V} (p_{i}-\bar{p}^k_i) \alpha_{i}^k &&  \label{eq:wta-ptp:of}\\
    \text{s.t.} & \;  \sum_{k\in K} \alpha^k_{i} \leq 1 \qquad && \forall\; i \in V \label{eq:wta-ptp:partition}\\
    & \; \sum_{i\in V} \alpha^k_{i}\leq b^k && \forall\; k \in K  \label{eq:wta-ptp:budget}\\
    & \; \sum_{i \in V} \bar p_{i}^k \alpha_{i}^k - \sum_{(i,j) \in A} c^k_{ij} z_{ij}^k  \geq 0 && \forall\; k \in K \label{eq:wta-ptp:wta}\\
     & \; ({\alpha}^k,{z}^k) \text{ is a route} && \forall\; k \in K \label{eq:wta-ptp:route}\\
    & \; {\alpha}^k\in \{0,1\}^{n+1}, {z}^k\in \{0,1\}^{|A|} && \forall\; k \in K.
\end{empheq}
\end{subequations}
The objective function~\eqref{eq:wta-ptp:of} represents the net profit of the platform to be maximized. The first set of constraints~\eqref{eq:wta-ptp:partition} imposes that each item is served by at most one carrier; the second ones~\eqref{eq:wta-ptp:budget} are the capacity constraints (equivalently one could consider the route duration limit); constraints~\eqref{eq:wta-ptp:wta} state that each carrier should have a nonnegative profit; finally, constraints~\eqref{eq:wta-ptp:route} ensure that ${z}^k$ is the incidence vector of a route that visits the depot and all customers $i$ such that $\alpha^k_i=1$. The value of WTA-PTP as calculated in~\eqref{eq:wta-ptp:of} provides an upper bound to BPFM. Indeed, given constraints~\eqref{eq:wta-ptp:wta}, in WTA-PTP, we assume carriers accept to serve the whole assigned bundle of items, as long as the associated net profit is non-negative. Instead, in BPFM, carriers determine the subset of assigned items maximizing their net profit. 
However, we note that, given any feasible solution $\bar{{\alpha}}$ of $(\mathsf{\textcolor{OliveGreen}{WTA\text{-}PTP}})$, a bilevel feasible solution can be \textit{recovered} by solving a PTP for each carrier, on the subset of items assigned to that carrier in the $(\mathsf{\textcolor{OliveGreen}{WTA\text{-}PTP}})$ solution. Indeed, in this way, we can compute the accepted subset of the items among the ones for which $\bar \alpha_i$ is $1$, and these subsets, associated with the optimal tours to serve them, represent a feasible solution of the BPFM.

Formulation~$(\mathsf{\textcolor{OliveGreen}{WTA\text{-}PTP}})$ can be easily generalized to the case in which the platform can decide also on the compensation paid to each carrier, in terms of the price margin the platform gains from each item. Indeed, by replacing variable $\alpha_i^k$ with $A_{mi}^k$ which is $1$ if item $i$ is served by carrier $k$ with a margin of $m$ for the platform, we obtain the following formulation, which gives us a valid upper bound for BPMD:
{\small\begin{subequations}\label{eq:wta-ptp-md}
\begin{empheq}[left=(\mathsf{\textcolor{OliveGreen}{WTA\text{-}PTP\text{-}MD}})\empheqlbrace]{align}
			\max\limits_{{A},{z}} &  \; \sum_{k \in K}\sum_{i \in V}\sum_{m \in M} p_{mi}A_{mi}^k &&  \label{eq:wta-ptp-md:of}\\
			\text{s.t.} & \;  \sum_{k\in K}\sum_{m \in M} A^k_{mi} \leq 1 \qquad && \forall\; i \in V \label{eq:wta-ptp-md:partition}\\
			& \; \sum_{i\in V}\sum_{m \in M} A^k_{mi}\leq b^k && \forall\; k \in K  \label{eq:wta-ptp-md:budget}\\
			& \; \sum_{i \in V}\sum_{m \in M} \bar p_{mi}^k A_{mi}^k - \sum_{(i,j) \in A} c^k_{ij} z_{ij}^k  \geq 0 && \forall\; k \in K \label{eq:wta-ptp-md:wta}\\
			& \; (\sum_{m \in M} {A}_m^k,{z}^k) \text{ is a route} && \forall\; k \in K \label{eq:wta-ptp-md:route}\\
			& \; {A}_m^k\in \{0,1\}^{n+1}, {z}^k\in \{0,1\}^{|A|} && \forall\; m \in M, k \in K.
		\end{empheq}
\end{subequations}} 
Let us now consider a different setting. If a profit-sharing between the platform and the set of carriers is considered, a different single-level formulation is obtained. This situation arises when dealing with the so-called Urban Consolidation Centers (UCC), or City Logistics Centers, i.e., logistics facilities strategically located within urban areas to optimize the efficiency of last-mile deliveries \citep{handoko2015}. In this case, the platform corresponds to the alliance of all carriers, and the objective function is the total profit of this alliance. The corresponding single-level formulation is then:
\begin{subequations}\label{eq:ucc-ptp}
\begin{empheq}[left=(\mathsf{\textcolor{OliveGreen}{UCC\text{-}PTP}})\empheqlbrace]{align}
    \max\limits_{{\alpha},{z}} &  \; \sum_{k \in K}\left[\sum_{i \in V} \bar{p}^k_i\alpha_i^k - \sum_{(i,j) \in A} c^k_{ij} z_{ij}^k\right]  &&  \label{eq:ucc-ptp:of}\\    
    \text{s.t.} & \;  \sum_{k\in K} \alpha^k_{i} \leq 1 \qquad && \forall\; i \in V \label{eq:ucc-ptp:partition} \\
    & \; \sum_{i\in V} \alpha^k_{i}\leq b^k && \forall\; k \in K \label{eq:ucc-ptp:budget} \\
    & \; ({\alpha}^k,{z}^k) \text{ is a route} && \forall\; k \in K \label{eq:ucc-ptp:route}\\
    & \; {\alpha}^k\in \{0,1\}^{n+1}, {z}^k\in \{0,1\}^{|A|} && \forall\; k \in K.
\end{empheq}
\end{subequations}
This single-level formulation, which we define as UCC-PTP, can be used to obtain a lower bound on the profit of our delivery platform in the setting we study. Indeed, the solution of UCC-PTP is a feasible solution for the BPFM, as well as for the BPMD.

With the following example, we better clarify the relationships between the two single-level problems presented above and the bilevel problem we focus on.
\begin{figure}[ht]
	\centering
	\includegraphics[width=0.6\textwidth]{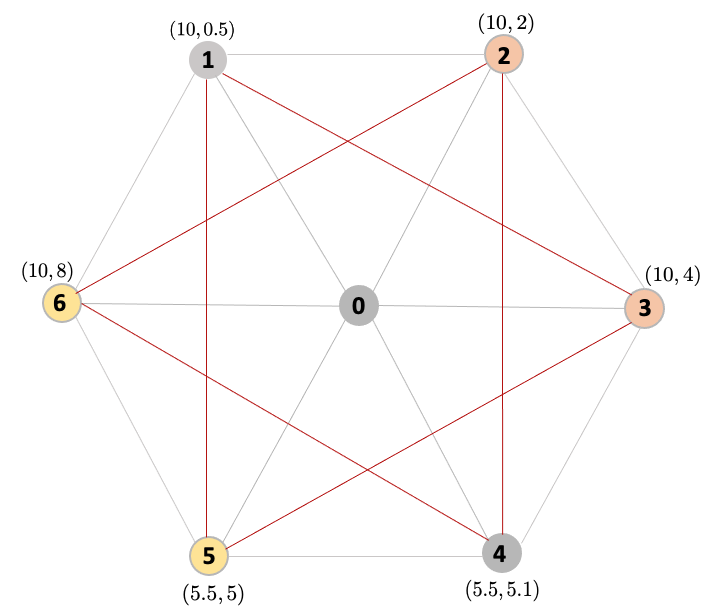}
	\caption{Example graph. Weights of gray/red edges are $0.5/1$ respectively. Labels next to each vertex $i$ display $(p_i,\bar p^k_i)$ for all $k$.} 
	\label{fig:1}
\end{figure}

Consider the complete graph with $7$ nodes in Figure~\ref{fig:1}, where node $0$ is the depot. Assume that there are two carriers $a$ and $b$, and that for all $k\in\{k_a,k_b\}$:
\begin{itemize}
	\item $c_{0j}^{k} = 0.5$ for all $j \in V$, $c_{12}^{k} = c_{23}^{k} = c_{34}^{k} = c_{45}^{k} = c_{51}^{k} = 0.5$, $c_{13}^{k} = c_{14}^{k} = c_{24}^{k} = c_{25}^{k} = c_{35}^{k} = 1$, and $c^k_{ij} = c_{ji}^k$ for all $(i,j) \in A$;
	\item $b^{k}=2$;
	\item $(p_1,\bar p_1^{k})=(10,0.5)$, $(p_2,\bar p_2^{k})=(10,2)$, $(p_3,\bar p_3^{k})=(10,4)$, $(p_4,\bar p_4^{k})=(5.5,5.1)$, $(p_5,\bar p_5^{k})=(5.5,5)$, $(p_6,\bar p_6^{k})=(10,8)$.
\end{itemize}
 In the WTA-PTP setting, the leader would assign to carriers $k_a$ and $k_b$ items $1$, $2$, $3,$ and $6$ (two to each carrier), predicting a total profit of $25.5$. The platform would exclude items $4$ and $5$ because they correspond to the smallest margins of $0.4$ and $0.5$, respectively. Let us now consider the BPFM as described above. The carrier who is assigned item $1$, together with any other item among $2$, $3,$ and $6$, would not deliver it because serving only the other offered item ($2$, $3$, or $6$) would produce a higher profit than serving the assigned pair. Thus, the profit of the platform associated with their optimal WTA-PTP solution is $16$ (value of the \textit{recovered} bilevel feasible solution).

In the UCC-PTP setting, the UCC would assign to carrier $k_a$ items $3$, and $4$, and to carrier $k_b$ items $5,$ and $6$ (or vice-versa), with a UCC-PTP value of $19.1$. This is a feasible solution in the BPFM, with a value of $8.9,$ but it is not the optimal one. Indeed, the optimal solution of BPFM would assign items $2, 3, 5$ and $6$ (two to each carrier, see the colors of the nodes in Figure~\ref{fig:1} for a graphical representation), yielding a total profit of $16.5$.

This example illustrates that the problem we study is bilevel in nature and that, to avoid misprediction of the true profit of the leader, it is crucial to integrate the optimal followers' response inside of the optimization problem, as it is done with the BPMF. Approximating this optimal response, and either replacing it with WTA assumption or considering the carriers as part of the platform leads to a suboptimal decision for the platform. Nevertheless, the values of WTA-PTP and UCC-PTP can be used to derive an upper and a lower bound on the net profit of the leader in the bilevel setting, respectively.

\section{The Bilevel PTP with Fixed Margins}\label{sec:fixed_margins}

In this section, we introduce a formulation for the BPFM, first, with a limit on the number of packages, and then, in Subsection~\ref{sub:tmax}, with a limit on the duration of the route. 

We recall that in the BPFM the leader does not decide on the compensations to be paid to the carriers, i.e., $\bar p_i^k$ is fixed and given a-priori.
To model the leader's decision on the proposal to each carrier, we use the binary decision variable $x^k_{i}$ for each $i\in V$, $k\in K$, which takes the value 1 if the platform assigns customer $i$ to carrier $k$, i.e., $i\in P_k$.
In addition, we define the lower-level binary variable $y^k_{i}$ for each $i\in V_0$ and $k\in K$, to model the acceptance decision of the carriers. $y^k_{i}$ is 1 if carrier $k$ accepts to serve customer $i$, i.e., $i \in Q_k$; in particular, $y^{k}_0$ is equal to 1 in case carrier $k$ accepts to make at least one delivery.
Finally, we consider the set of lower-level binary variables $z^k_{ij}$ for all $(i,j)\in A$ and $k\in K$ (already introduced for formulations~\eqref{eq:wta-ptp} and \eqref{eq:ucc-ptp}) to model the routing decisions of the carriers. In particular, $z^k_{ij}=1$ if arc $\left(i,j\right) \in A$ is traversed by carrier $k$.

Then, the BPFM formulation is as follows:
\begin{subequations}\label{eq:routing}
	 \begin{align}
		\max\limits_{{x},{y}} & \quad \sum_{i \in V} \sum_{k \in K} \left(p_i - \bar p^k_{i}\right) y^k_i&& \label{eq:routing-of}\\
		\text{s.t.} & \quad \sum_{k\in K} x^k_{i} \leq 1 \qquad && \forall\; i \in V \label{eq:routing-partition}\\
		& \quad {y}^k \in S^k_{\varphi}({x}^k) && \forall\; k \in K \label{eq:routing-valuefunction}\\
		& \quad {x}^k \in \{0,1\}^{n}, \; {y}^k \in \{0,1\}^{n+1}&& \forall\; k \in K, \label{eq:sets}
	\end{align}
\end{subequations}
where $S^k_{\varphi}({x}^k)$ is the set of optimal solutions of the $k$-th follower problem, which, for a given $\tilde{{x}}^k$, is formulated as:
\begin{subequations}\label{eq:routingLL}\
	\begin{align}
		\varphi^k(\tilde{{x}}^k) = \quad \max\limits_{{y},{z}} & \quad\sum_{i \in V} \bar p^k_{i} y^k_i - \sum_{(i,j) \in A} c^k_{ij} z^k_{ij} && \label{eq:routingLL-of}\\
		\text{s.t.} & \quad y^k_{i} \leq \tilde x^k_i &&\forall\; i \in V \label{eq:routingLL-interdiction}\\
        & \quad \sum_{i\in V} y^k_{i} \leq b^k && \label{eq:routingLL-budget} \\
		& \quad ({y}^k,{z}^k) \text{ is a route} && \label{eq:routingLL-route}\\
		& \quad {y}^k\in \{0,1\}^{n+1}, {z}^k\in \{0,1\}^{|A|}.&&
	\end{align}
\end{subequations}
Constraints~\eqref{eq:routing-partition} state that each item is offered to at most one carrier. Constraints~\eqref{eq:routing-valuefunction} ensure that the solution ${y}^k$ returned by the $k$-th follower is an \textit{optimal response} with respect to the set of items offered to her by the leader.
The objective function~\eqref{eq:routingLL-of} of follower $k$ is the difference between the sum of the delivered items' compensations and the total travel cost (length of the route that visits all customers accepted by the carrier). Constraints~\eqref{eq:routingLL-interdiction} link the decisions of the leader with the ones of the follower and establish that an item $i$ can be delivered by carrier $k$ only if it is offered to her. Constraints~\eqref{eq:routingLL-budget} impose that carrier $k$ can accept to serve at most $b^k$ items. Finally, constraint~\eqref{eq:routingLL-route} states that ${z}^k$ is the incidence vector of a route that visits the depot and all customers $i$ such that $y^k_i=1$. More in detail, constraint~\eqref{eq:routingLL-route} is given by:
\begin{subequations}
	\begin{alignat}{3}
		\sum\limits_{(i,j) \in \delta^+(i)} z_{ij}^k \; = \; & y_i^k &\quad\forall\;i \in V\label{eq:route1}\\
		\sum\limits_{(i,j) \in \delta^-(i)} z_{ji}^k \; = \; & y_i^k &\quad\forall\;i \in V\label{eq:route2}\\
		\sum\limits_{i \in S , j \in S} z_{ij}^k \; \leq\; & \sum_{i \in S} y_i - y_h &\quad\forall\;S \subseteq V, |S|\geq 2, h \in S. \label{eq:subtour}
	\end{alignat}
\end{subequations}
The sets of equalities~\eqref{eq:route1}--\eqref{eq:route2} impose that one arc enters and leaves each visited vertex.
The set of exponentially many inequalities~\eqref{eq:subtour} ensures subtour elimination and connection to the depot. 

In what follows, in order to derive a single-level reformulation of the BPFM formulation~\eqref{eq:routing}, we 
propose two approaches. The first one, presented in Subsection~\ref{sub:BPFM-valuefunct}, is the value function reformulation approach typically used in bilevel literature (see Online Appendix~A for more details). The relationship of the so-obtained single-level formulation with the WTA-PTP and UCC-PTP formulations is discussed in Subsection~\ref{sub:comparison-single}, and an alternative formulation is obtained by projecting out the routing variables in Subsection~\ref{sub:proj}. A variant of the problem when considering a limit on the route duration instead of the capacity constraints is presented in Subsection~\ref{sub:tmax}. Then, in Subsection~\ref{sub:BPFM-nogood}, the second reformulation approach, based on bilevel \textit{no-good} cuts \citep{tahernejad2020valid} exploiting the binary nature of the upper-level decisions, is discussed.

\subsection{Value function reformulation}\label{sub:BPFM-valuefunct}
As explained in Online Appendix~A, one possible way to reformulate optimistic bilevel problems is through the so-called value function approach, which, for model~\eqref{eq:routing}, leads to the following reformulation:
\begin{subequations}\label{eq:routing_single1}
	\begin{align}
   	\max\limits_{{x},{y},{z}} & \quad \sum_{i \in V} \sum_{k \in K} \left(p_i - \bar p^k_{i}\right) y^k_i&&  \label{eq:routing_single-of1}\\
		\text{s.t.} & \quad  \sum_{k\in K} x^k_{i} \leq 1 \qquad && \forall\; i \in V \label{eq:routing_single-partition1}\\
        & \quad \sum_{i\in V} x^k_{i} \leq b^k && \forall\; k \in K \label{eq:routing_single-budget1} \\
	   \quad \quad & \quad \sum_{i \in V}  \bar p^k_{i} y^k_i  - \sum_{(i,j) \in A} c^k_{ij} z^k_{ij} \geq \varphi^k({x}^k) && \forall\; k \in K \label{eq:routing_single-valuefunction1}  \\
	   & \quad y^k_{i}  \leq  x^k_i && \forall\; i \in V,\; k \in K  \label{eq:routing_single-interdiction1}\\
	   & \quad ({y}^k,{z}^k) \text{ is a route} && \forall\; k \in K \label{eq:routing_single-route1}\\
	   &  \quad {x}^k{\in \{0,1\}^{n}},{y}^k\in \{0,1\}^{n+1},{z}^k\in \{0,1\}^{|A|} && \forall\; k \in K.\label{eq:routing_single-sets1}
	\end{align}
\end{subequations}
In formulation~\eqref{eq:routing_single1}, we replace constraints
\begin{equation}\label{eq:y_routing}
    \sum_{i\in V} y^k_{i} \leq b^k \quad \forall\; k \in K
\end{equation}
of the classic value function reformulation of the BPFM~\eqref{eq:routing} with constraints~\eqref{eq:routing_single-budget1}, which, combined with \eqref{eq:routing_single-interdiction1}, imply~\eqref{eq:y_routing}. Constraints~\eqref{eq:routing_single-budget1} on ${x}$ are indeed stronger than~\eqref{eq:y_routing}, and this helps in the solution of the single-level formulation~\eqref{eq:routing_single1}.

The difficulty of solving formulation~\eqref{eq:routing_single1} lies in the value-function constraints~\eqref{eq:routing_single-valuefunction1}: the function $\varphi^k({x}^k)$ is non-convex and non-continuous. Thus, in order to derive a single-level MIP formulation of the problem, we further analyze this function, which we try to convexify. Under the assumption that arc weights $c^k_{ij}$ satisfy the triangle inequality (in order to ensure that $z^k_{ij}=1$ implies $y^k_i=y^k_j=1$ for all $k$), we derive the following result.

\begin{prop}\label{prop}
For any $k\in K$, given a vector $\tilde{{x}}^k \in \{0,1\}^{n}$  satisfying constraints~\eqref{eq:routing_single-partition1}--\eqref{eq:routing_single-budget1}, there always exists an optimal solution of the following problem, which is also optimal for $\varphi^k(\tilde{{x}}^k)$:
\begin{subequations}\label{eq:routingLL1}
    \begin{align}
        \bar \varphi^k(\tilde{{x}}^k) = \quad \max\limits_{{y},{z}} & \quad \sum_{i \in V}  \bar p^k_{i} y^k_i \tilde x^k_i - \sum_{(i,j) \in A} c^k_{ij} z^k_{ij} && \label{eq:routingLL1-of}\\
        \text{s.t.} &\quad ({y}^k,{z}^k) \text{ is a route} && \label{eq:routingLL1-route}\\
        & \quad {y}^k\in \{0,1\}^{n+1}, {z}^k\in \{0,1\}^{|A|}. && \label{eq:routingLL1-sets}
    \end{align}
\end{subequations}
\end{prop}
\begin{proof}
    Let $\tilde{{x}}^k \in \{0,1\}^{n}$ be a vector satisfying constraints~\eqref{eq:routing_single-partition1}--\eqref{eq:routing_single-budget1}. Let us consider a given $k \in K$ and let $P_k\subseteq V$ be the subset of vertices associated with $\tilde{{x}}^k$, i.e., $P_k=\{i\in V: \tilde x^k_i=1\}$, with $|P_k| = b^k$.
    Being $S_{\bar \varphi}^k(\tilde{{x}}^k)$ the set of optimal solutions of problem~$\bar \varphi^k(\tilde{{x}}^k)$, we want to prove that there exists $(\hat{{y}}^k, \hat{{z}}^k) \in S_{\bar \varphi}^k(\tilde{{x}}^k)$
    such that $a)$ its objective function value is $ \sum\limits_{i \in P_k} \bar p^k_{i} \hat y^k_i - \sum\limits_{(i,j) \in A} c^k_{ij} \hat z^k_{ij}$, and $b)$ $\hat y^k_i=0$ for all $i \in V\setminus P_k,$ despite the relaxation of constraints~\eqref{eq:routingLL-interdiction} in formulation~\eqref{eq:routingLL1}. If this is true, constraint~\eqref{eq:routingLL-budget} will also hold because of constraints~\eqref{eq:routing_single-budget1}, which impose that $|P_k| = b^k$.
    The reasoning is the following. If an $i' \in V\setminus P_k$ exists s.t.\ $\hat y^k_{i'}=1$, then the compensation collected from $i'$ would be 0 as $\tilde x^k_{i'}=0$. Also, for the triangle inequality, going directly from the predecessor to the successor of $i'$ in the optimal solution is cheaper (or at most has the same cost) than going through $i'$. Thus, either the solution visiting $i'$ is not optimal, or there exists a solution that does not visit $i'$ with the same objective function value. This procedure can be iterated over all $i \in V\setminus P_k$.
\end{proof}
	
According to Proposition \ref{prop}, the problem of follower $k$ could be solved by considering the entire graph~$\cal{G}$ and multiplying the compensation associated with each customer $i \in V$ by the value $\tilde{x}_i^k$, representing the assignment made by the leader. In this way, in case $\tilde{x}_i^k=0$, customer $i$ would not be visited in the optimal solution of follower $k$.

Let $P^k_{ext}$ denote the set of all the extreme points $({y}^k,{z}^k)$ of the convex hull of the follower's feasible solutions space determined by constraints~\eqref{eq:routingLL1-route}--\eqref{eq:routingLL1-sets}. It holds that:
$$ \varphi^k({x}^k) = \max_{(\hat{{y}}^k, \hat{{z}}^k) \in P^k_{ext}} \left\{\sum_{i \in V}  \bar p^k_{i}  \hat y^k_i x^k_i - \sum_{(i,j) \in A} c^k_{ij} \hat z^k_{ij}\right\}, $$
which is a convex function in ${x}^k$. 

We notice that, in terms of routes $T \in \mathcal{T},$ problem~\eqref{eq:routingLL1} can be restated as:
\[
\bar  \varphi^k(\tilde{{x}}^k) =  \max\limits_{T \in \mathcal{T}} \left\{ \sum_{i \in V(T)} \bar p_i^k\tilde x^k_i - C^k(T)  \right\}.
\]

In other words, constraints~\eqref{eq:routing_single-valuefunction1} can be replaced by constraints
\begin{equation}\label{eq:convexT}
    \sum_{i \in V}  \bar p^k_{i}  y^k_i - \sum_{(i,j) \in A} c^k_{ij} z^k_{ij} \geq \sum_{i \in V(T)} \bar p_i^k x^k_i - C^k(T)  \qquad \forall\; T \in \mathcal{T}, \; k \in K, 
\end{equation}
obtaining the following single-level reformulation of problem~\eqref{eq:routing_single1}:
\begin{subequations}\label{eq:routing_single}
	\begin{empheq}[left=(\mathsf{\textcolor{OliveGreen}{BPFM}})\empheqlbrace]{align}
    	\max\limits_{{x},{y},{z}} & \quad \sum_{i \in V} \sum_{k \in K} \left(p_i - \bar p^k_{i}\right) y^k_i&&  \label{eq:routing_single-of}\\
		\text{s.t.} & \quad \sum_{k \in K} x^k_i \leq 1 && \forall\; i \in V \\
        & \quad \sum_{i \in V} x^k_i \leq b^k && \forall\; k \in K \label{eq:routing_single-budget}\\
        & \quad \sum_{i \in V}  \bar p^k_{i}  y^k_i - \sum_{(i,j) \in A} c^k_{ij} z^k_{ij} \geq \sum_{i \in V(T)} \bar p_i^k x^k_i - C^k(T) && \forall\; T \in \mathcal{T}, \; k \in K \label{eq:routing_single-valuefunction}  \\
		& \quad y^k_{i}  \leq  x^k_i && \forall\; i \in V,\; k \in K  \label{eq:routing_single-interdiction}\\
		& \quad ({y}^k,{z}^k) \text{ is a route} && \forall\; k \in K \label{eq:routing_single-route}\\
& \quad {x}^k\in \{0,1\}^{n},{y}^k\in \{0,1\}^{n+1},{z}^k\in \{0,1\}^{|A|} && \forall\; k \in K,\label{eq:routing_single-sets}
    \end{empheq}
\end{subequations}
to which we will refer in the rest of the paper when considering the BPFM.

\subsection{Relationship with WTA-PTP and UCC-PTP formulations}\label{sub:comparison-single}
The $(\mathsf{\textcolor{OliveGreen}{WTA\text{-}PTP}})$ formulation and $(\mathsf{\textcolor{OliveGreen}{UCC\text{-}PTP}})$ formulation are respectively a relaxation and a restriction of the BPFM formulation. 
On the one hand, the WTA-PTP formulation~\eqref{eq:wta-ptp} can be obtained from formulation~\eqref{eq:routing_single} by setting ${x}={y}={\alpha}$ and replacing constraints~\eqref{eq:routing_single-valuefunction} with the weaker constraints~\eqref{eq:wta-ptp:wta}. 
On the other hand, to see the relationship with the UCC-PTP formulation, we set ${x}={\alpha}$, which gives us a valid partitioning of the items. Then, we set ${y}={\alpha}$ as a follower response. In order to prove that $({x},{y})$ is a feasible bilevel solution (and thus provides a lower bound for BPFM), it remains to prove that ${y}$ is the optimal followers' response for the given ${x}={\alpha}$. This follows from the model~\eqref{eq:ucc-ptp} itself because, once ${\alpha}$ is fixed, model~\eqref{eq:ucc-ptp} separates into $|K|$ independent subproblems, each of them corresponding to the lower-level problem~\eqref{eq:routingLL} for a given ${x}={\alpha}$. 

\subsection{\texorpdfstring{Projecting out the \boldmath{$z$} variables}{Projecting out the z variables}}\label{sub:proj}
In this section, we present a new formulation for the BPFM derived from projecting out the ${z}$ variables in the value-function reformulation presented above. We introduce the new continuous variables $\theta^k$, for each $k\in K$, which represent the cost of the route followed by carrier $k$. In this case, problem~$(\mathsf{\textcolor{OliveGreen}{BPFM}})$ becomes
\begin{subequations}\label{eq:routing-proj}
    \begin{empheq}[left=\hspace*{-0.5cm}(\mathsf{\textcolor{OliveGreen}{BPFM\text{-}z}})\empheqlbrace]{align}
        \max\limits_{{x}, {y}, {\theta}} & \quad \sum_{i \in V} \sum_{k \in K} \left(p_i - \bar p^k_{i}\right) y^k_i&&  \label{eq:routing-proj-of}\\
        \text{s.t.} & \quad \sum_{k\in K} x^k_{i} \leq 1 && \forall\; i \in V \label{eq:routing-proj-partition}\\
        & \quad {\sum_{i\in V} x^k_{i} \leq b^k} && \forall\; k \in K \label{eq:routing-proj-budget} \\
        & \quad \sum_{i \in V}  \bar p^k_{i} y^k_i  - \theta^k \geq \varphi^k({x}^k) && \forall\; k \in K \label{eq:routing-proj-valuefunction}  \\
        & \quad \theta^k \geq c^k_{\mathsf{TSP}}({y}^k) &&\forall\;k \in K \label{eq:routing-proj-theta}\\
        & \quad y^k_{i}  \leq  x^k_i && \forall\; i \in V, k \in K  \label{eq:routing-proj-interdict}\\
        &  \quad {x}^k \in \{0,1\}^n, {y}^k\in \{0,1\}^{n+1},\theta^k \in \mathbb{R} && \forall\; k \in K\label{eq:routing-proj-sets}
    \end{empheq}
\end{subequations}

where $c^k_{\mathsf{TSP}}({y}^k)$ is the cost of the optimal route associated with vector ${y}^k$ (TSP standing for Travelling Salesman Problem), and $\varphi^k({x}^k)$ is the optimal solution value of the $k$-th follower problem, which, for a given $(\tilde{{x}}^k)$ is formulated as in~\eqref{eq:routingLL}.

According to Proposition~\ref{prop}, constraints~\eqref{eq:routing-proj-valuefunction} can be replaced by:
\begin{subequations}
\begin{equation}\label{eq:routing-proj-cuts}
    \sum_{i \in V} \bar p_i^k y_i^k - \theta^k  \geq \sum_{i \in V(T)} \bar p_i^k x^k_i - C^k(T) \quad \forall\;T \in \mathcal{T}, k \in K,
\end{equation}
whereas constraints~\eqref{eq:routing-proj-theta} can be replaced by the cuts:
\begin{equation} \label{eq:nogood-routing}
    \theta^k  \geq c^k_{\mathsf{TSP}}(V(T)) \left[\sum_{i \in V(T)} y^k_i - |V(T)| + 1\right] \quad \forall\; T \in \mathcal{T}, k \in K.
\end{equation}
\end{subequations}

Additionally, we can add to the formulation the following strengthening inequalities for all $k \in K$ in order to provide a non-trivial lower bound on the value of variables $\theta^k$:
\begin{equation}\label{eq:thetabound}
    \theta^k \geq \sum_{i \in V} d_i^k y_i^k,
\end{equation}
where $d_i^k=\max\{\min\limits_{j \in \delta^-(i)}c^k_{ji},\min\limits_{j \in \delta^+(i)}c^k_{ij}\}$. These inequalities state that the cost of the route associated with ${y}^k$ cannot be lower than the sum of the costs of the arcs of minimum cost incident to all the visited nodes.

\subsection{\texorpdfstring{Bounded route duration variant}{Bounded route duration variant}}\label{sub:tmax}
In the setting considered above, each carrier cannot deliver more than $b^k$ packages (constraints~\eqref{eq:wta-ptp:budget}, \eqref{eq:ucc-ptp:budget}, \eqref{eq:routingLL-budget}, \eqref{eq:routing_single-budget}, \eqref{eq:routing-proj-budget}). Another common practical setting is the one where the limitation is instead imposed on route duration, thus having a time limit $t_{max}^k$. To model this case, we remove the constraints on the maximum number of packages and we add the following constraint:

\begin{equation}\label{eq:pricing-tmax}
	\quad \sum\limits_{(i,j)\in A} c^k_{ij} z_{ij}^k \leq t_{max}^k
\end{equation}
in each lower level. The considerations on the bilevel nature of the problem also apply in this case.

The formulation presented in Section~\ref{sub:proj}, where ${z}$ variables are projected out, has to include constraint
\begin{equation}\label{eq:proj-tmax} 
    \theta^k  \leq t_{max}^k \quad \forall k \in K
\end{equation}
in the single-level formulations, and constraint~\eqref{eq:pricing-tmax} in each lower level.

\subsection{Reformulation based on no-good cuts}\label{sub:BPFM-nogood}
An alternative to the value function approach for obtaining a single-level reformulation of model~\eqref{eq:routing} consists in introducing no-good cuts based on the solution of the PTPs associated with the carriers.

Consider a set of items $\mathcal{V}^k$, such that, when they are offered to carrier $k$ ($\alpha_i^k = 1$ for all $i \in \mathcal{V}^k$), not all are accepted. In other words, there exists a proper subset of items within $\mathcal{V}^k$ that constitutes an optimal PTP solution for carrier $k$.
Using the ${\alpha}$ variables introduced for the WTA-PTP formulation~\eqref{eq:wta-ptp}, we can obtain the following reformulation of BPFM with no-good cuts:
\begin{subequations}\label{eq:BPFM2}
	\begin{empheq}[left=(\mathsf{\textcolor{OliveGreen}{BPFM'}})\empheqlbrace]{align}
    	\max\limits_{{\alpha}} & \; \sum_{i \in V} \sum_{k \in K} \left(p_i - \bar p^k_{i}\right) \alpha^k_i&&  \\
		\text{s.t.} & \; \sum_{k \in K} \alpha^k_i \leq 1 && \forall\; i \in V \\
        & \; \sum_{i \in V} \alpha^k_i \leq b^k && \forall\; k \in K \label{eq:BPFM2:budget} \\
        & \; \sum_{i \in \mathcal{V}^k}  (1 - \alpha^k_i) +  \sum_{i \not \in \mathcal{V}^k} \alpha^k_{i} \geq 1 && \forall\; k \in K, \; \mathcal{V}^k \subset V \label{eq:BPFM2:nogood} \\
		& \; {\alpha}^k\in \{0,1\}^{n} && \forall\; k \in K.
    \end{empheq}
\end{subequations}
Constraints~\eqref{eq:BPFM2:nogood} are the no-good cuts and they are exponentially many. Note that for any given subset of items whose incidence vector is given by ${\alpha}^k$, in order to determine whether it coincides with a $k$-th carrier's optimal response, one has to solve the lower-level PTP problem~\eqref{eq:routingLL}, obtaining the set of accepted parcels ${y}^k$. If the vectors ${\alpha}^k$ and ${y}^k$ do not coincide, the no-good cut has to be added to the model. Hence, this is an alternative way of restating the bilevel problem as a single-level reformulation. However, it is worth noting that the no-good cuts are known to be weak, only cutting off one point at a time.

\section{The Bilevel PTP with Margin Decisions}\label{sec:variable_margins}
In this section, we turn our attention to the BPMD and present two different formulations for this problem in Section~\ref{sub:agg} and Section~\ref{sub:disagg}, respectively. In Section~\ref{sub:comparison}, we compare these two bilevel formulations, referring to their value function reformulations.

As already discussed, in the BPMD we assume that, in addition to assignment decisions, the leader decides also the margin $m \in M$ to gain from each item~$i$ delivered by carrier~$k$.

Let $m_{\text{min}}$ and $m_{\text{max}}$ be the minimum and maximum margin, respectively. In order to model the leader's choice among the different margins $m \in M$, we have considered the following two alternative options:
\begin{itemize}
    \item[$i)$] 
    We define a new upper-level binary variable for each carrier, margin, and item, $X_{mi}^k$, which takes value $1$ if $i\in P_k$ and the selected margin is $m$ for item $i\in V$, and $0$ otherwise. As $\sum\limits_{m \in M} X_{mi}^k = x_i^k$ for all $i \in V$ and $k \in K$, we can discard the variables ${x}$. However, we still use variables ${y}$ as defined for the BPFM.
    \item[$ii)$] 
    We define also the lower-level binary \textit{disaggregated} variable $Y_{mi}^k$, which substitutes the previous decision variable $y^k_i$, and takes value $1$ if carrier $k$ accepts to serve customer~$i$ (i.e., $i \in Q_k$) with margin level~$m$. In particular $Y^k_{m0} = 1$ for an arbitrarily selected $m$ if carrier $k$ accepts to make at least one delivery.
\end{itemize}

In both cases, we model the routing decisions with the same binary variables ${z}$ as defined for the BPFM. The decision variables of alternative $i)$ lead to what we call the \textit{aggregated formulation}, in contrast to alternative $ii)$ which leads to a \textit{disaggregated formulation}.
We prove in Section~\ref{sub:comparison} that these two formulations provide the same bounds.

Before presenting the formulations we observe that if the selected margin is $m$, the corresponding net profit for the leader will be $p_{mi}y_i^k$, where $p_{mi} = m \cdot p_i$. Furthermore, in this context, we denote by $\bar p_{mi} = p_i-p_{mi}$ the compensation paid to follower $k$ when the selected margin is $m$.

\subsection{Aggregated formulation}\label{sub:agg}

Using the decision variables of alternative $i)$, the upper-level variable $x_i^k$ can be replaced by $\sum\limits_{m \in M} X_{mi}^k.$
Furthermore, the upper-level objective function, representing the profit of the leader to be maximized, reads
$$\sum_{k \in K} \sum_{i \in V}\sum_{m \in M} p_{mi} X_{mi}^k y_i^k.$$
We can linearize this function by using Fortet's inequalities \citep{fortet1960applications}, i.e., a special case of the McCormick inequalities for products of binary variables. These inequalities define the convex envelopes of the bilinear terms $X^k_{mi} y^k_i$. In order to do this, we have to introduce additional binary variables $w_{mi}^k$ defined as $X_{mi}^k y^k_i$ and insert the Fortet's inequalities~\eqref{eq:pricing1-mc1}--\eqref{eq:pricing1-mc3} in the upper-level problem formulation obtaining the following bilevel problem:
\begin{subequations}\label{eq:pricing1}
    \begin{align}
        \max\limits_{{X},{w},{y}} &  \quad \sum_{k \in K} \sum_{i \in V} \sum_{m \in M} p_{mi} w_{mi}^k&&  \label{eq:pricing1-of}\\
        \text{s.t.} & \quad  \sum_{k\in K}\sum_{m \in M} X^k_{mi} \leq 1 \qquad && \forall\; i \in V \label{eq:pricing1-partition}\\
        & \quad X_{mi}^k + y^k_i \le w_{mi}^k + 1 && \forall\; m \in M, i \in  V, k \in K \label{eq:pricing1-mc1} \\
        & \quad  w_{mi}^k \le X_{mi}^k   && \forall\; m \in M, i \in  V, k \in K \label{eq:pricing1-mc2}\\
        & \quad  w_{mi}^k  \le y^k_i  && \forall\; m \in M, i \in  V, k \in K \label{eq:pricing1-mc3}\\
        &  \quad {y}^k \in S^k_{\Phi}({X}^k) && \forall\; k \in K \label{eq:pricing1-valuefunction}\\
        & \quad {X}_m^k, {w}_m^k\in \{0,1\}^{n} && \forall\; m \in M, k \in K\\
        & \quad {y}^k \in \{0,1\}^{n+1} && \forall\; k \in K
    \end{align}
\end{subequations}
with $S^k_{\Phi}({X}^k)$ the set of optimal solutions of the $k$-th follower problem, which, for a given $\tilde{{X}}^k$, reads:
\begin{subequations}\label{eq:pricing1LL}
    \begin{align}
        \Phi^k(\tilde{{X}}^k) = \; \max\limits_{{y},{z}} & \;\sum_{i \in V} \sum_{m \in M}\bar p_{mi} \tilde X_{mi}^k y^k_i - \sum_{(i,j) \in A} c^k_{ij} z^k_{ij} &&  \label{eq:pricing1LL-of}\\
        \text{s.t.} & \quad  y^k_{i} \leq \sum_{m \in M}\tilde X^k_{mi} &&\forall\; i \in V  \label{eq:pricing1LL-interdiction}\\
        & \quad \sum_{i\in V} y^k_{i}\leq b^k && \label{eq:pricing1LL-budget} \\
        & \quad ({y}^k,{z}^k)  \text{ is a route} && \label{eq:pricing1LL-route}\\
        & \quad {y}^k\in \{0,1\}^{n+1},{z}^k\in \{0,1\}^{|A|}. && \label{eq:pricing1LL-sets}
    \end{align}
\end{subequations} 

We note that, if $\sum\limits_{m \in M} X_{mi}^k = 0,$ then $X_{mi}^k=0$ for all $m$, and hence $y_i^k$ will be $0$ (from constraints~\eqref{eq:pricing1LL-interdiction}), so the upper-level objective function, as well as the first term of the lower-level objective function, will also be 0. Thus, we can replace constraints~\eqref{eq:pricing1-mc1} by: 
\begin{subequations}
\begin{equation}\label{eq:pricing1-mc1bis}
    y^k_i \le w_{mi}^k + \sum_{u \in M:u \neq m} X_{ui}^k  \quad \quad \forall\;m \in M, i \in  V, k \in K.
\end{equation}

Furthermore, to strengthen the Linear Programming (LP) relaxation of the problem, we add the following constraint, implicitly satisfied for binary solutions, but not necessarily for LP solutions:
\begin{equation}\label{eq:pricing1-yw}
    y^k_i = \sum_{m \in M} w_{mi}^k  \quad \quad \forall\;i \in  V, k \in K.
\end{equation}
\end{subequations}
\indent A single-level reformulation of the proposed bilevel problem may be obtained using the value function reformulation. Also in this case, we restrict ${x}$ instead of ${y}$ when moving constraints~\eqref{eq:pricing1LL-budget} to the upper level, i.e., in the single-level value function reformulation we have constraints 
\begin{equation}\label{eq:x_bound}
    \sum_{i\in V} \sum_{m \in M}X_{mi}^k \leq~b^k \quad \forall \; k \in K.
\end{equation}
As in Section~\ref{sec:fixed_margins}, assuming parameters $c^k_{ij}$ satisfy the triangle inequality, the following result holds.
		
\begin{prop}\label{prop2}
    For any $k\in K$, given a vector $\tilde{{X}}^k \in \{0,1\}^{n}$ satisfying constraints~\eqref{eq:pricing1-partition}--\eqref{eq:pricing1-mc3}, and \eqref{eq:x_bound}, there always exists an optimal solution of the following problem, which is also optimal for $\Phi^k(\tilde{{X}}^k)$:
    \begin{align}\label{subeq:problem}
        \bar \Phi^k(\tilde{{X}}^k) = \quad \max\limits_{T\in \mathcal{T}} \left\{ \sum_{i \in V(T)} \sum_{m \in M}\bar p_{mi}\tilde X_{mi}^k - C^k(T)\right\}.
    \end{align}
\end{prop}
\begin{proof}
Problem~\eqref{subeq:problem} is obtained from \eqref{eq:pricing1LL} by relaxing constraints~\eqref{eq:pricing1LL-interdiction} and \eqref{eq:pricing1LL-budget}.
When, for a given $i' \in V$, $\sum\limits_{m \in M}\tilde X^k_{mi'} = 1$, constraint~\eqref{eq:pricing1LL-interdiction} is implicitly satisfied, being the components of ${y}$ at most 1. Thus, being $S_{\bar \Phi}^k(\tilde{{X}}^k)$ the set of optimal solutions of problem~${\bar \Phi}^k(\tilde{{X}}^k)$, we want to prove that, if, for a given $i'$, $\sum\limits_{m \in M}\tilde X^k_{mi'} = 0$ there exists $(\hat{{y}}^k, \hat{{z}}^k) \in S_{\bar \Phi}^k(\tilde{{X}}^k)$ such that $\hat y^k_{i'}=0$. Assuming that $\hat y^k_{i'}=1$, the compensation collected from $i'$ would be 0 as $\tilde X^k_{mi'}=0$ for all $m$. Also, for the triangle inequality, going directly from the predecessor to the successor of $i'$ in the optimal solution is cheaper (or at most has the same cost) than going through $i'$. Thus, either the solution visiting $i'$ is not optimal, or there exists a solution that does not visit $i'$ with the same value of the objective function. Straightforwardly, as constraints~\eqref{eq:x_bound} hold for $\tilde{{X}}^k$, constraint~\eqref{eq:pricing1LL-budget} will hold for $\hat{{y}}^k$.
\end{proof}
Following the same approach as in Proposition \ref{prop}, we have the following single-level reformulation:
{\small\begin{subequations}\label{eq:pricing1_single}
   \begin{empheq}[left=\hspace*{-1.5cm}(\mathsf{\textcolor{OliveGreen}{BPMD}})\empheqlbrace]{align}
   \max\limits_{{X},{w},{y},{z}} &  \quad \sum_{k \in K} \sum_{i \in V} \sum_{m \in M} p_{mi} w_{mi}^k &&  \label{eq:pricing1_single-of}\\
        \text{s.t.} & \quad \eqref{eq:pricing1-partition}, \eqref{eq:pricing1-mc1bis}\text{--}\eqref{eq:pricing1-yw}, \eqref{eq:pricing1-mc2}\text{--}\eqref{eq:pricing1-mc3} \nonumber \\
        & \quad {\sum_{i\in V} \sum_{m \in M}X_{mi}^k\leq b^k} && \forall\; k \in K \label{eq:pricing1_single-budget} \\
        & \quad y_i^k \leq \sum_{m \in M}X_{mi}^k && \forall\;i \in  V, k \in K \label{eq:pricing1_single-interdiction}\\
        & \quad ({y}^k,{z}^k)  \text{ is a route} && \forall\;k \in K \label{eq:pricing1_single-route}\\
        & \sum_{i \in V} \sum_{m \in M} \bar p_{mi}w_{mi}^k - \sum_{(i,j) \in A}c^k_{ij} z^k_{ij} \geq &&\nonumber\\
        & \qquad \qquad \qquad \sum_{i \in V(T)} \sum_{m \in M}\bar p_{mi}X_{mi}^k - C^k(T) && \forall\; T \in \mathcal{T}, k \in K \label{eq:pricing1_single-valuefunction}\\
        & \quad  {X}_m^k, {w}_m^k \in \{0,1\}^{n} && \forall\;m \in M,k \in K\\
        & \quad {y}^k \in \{0,1\}^{n+1},  {z}^k \in \{0,1\}^{|A|} && \forall\; k \in K \label{eq:pricing1_single-sets}
    \end{empheq}
\end{subequations}}
where constraints~\eqref{eq:pricing1_single-valuefunction} replace the value function constraints
\begin{equation*}
    \sum_{i \in V} \sum_{m \in M} \bar p_{mi} w_{mi}^k - \sum_{(i,j) \in A} c^k_{ij} z^k_{ij}  \geq \Phi^k({X}^k) \quad \forall\;k \in K.
\end{equation*}

In the same way as we did in Section~\ref{sub:proj} for BPFM, by introducing the new real variable $\theta^k$, constraints~\eqref{eq:pricing1_single-route} and~\eqref{eq:pricing1_single-valuefunction} in $(\mathsf{\textcolor{OliveGreen}{BPMD}})$ can be replaced by \eqref{eq:nogood-routing} and
\begin{align}
    \sum_{i \in V} \sum_{m \in M}\bar p_{mi} w_{mi}^k - \theta^k  \geq \sum_{i \in V(T)} \sum_{m \in M}\bar p_{mi}X_{mi}^k - C^k(T) &  \quad\forall\;T \in \mathcal{T}, k\in K,  \label{eq:pricing1-proj-valuefunction}
\end{align}
respectively, obtaining the new single-level formulation $(\mathsf{\textcolor{OliveGreen}{BPMD\text{-}z}})$, which is reported in Online Appendix~B. 

Furthermore, also for BPMD, as discussed for BPFM in Section~\ref{sub:tmax}, we can consider a route duration limit instead of a bound on the number of packages to serve, by replacing constraints~\eqref{eq:pricing1_single-budget} with constraints~\eqref{eq:pricing-tmax} for $(\mathsf{\textcolor{OliveGreen}{BPMD}})$, or with constraints~\eqref{eq:proj-tmax} for $(\mathsf{\textcolor{OliveGreen}{BPMD\text{-}z}})$.

\subsection{Disaggregated formulation}\label{sub:disagg}
Using the disaggregated decision variables $Y_{mi}^k$ of alternative $ii)$ we obtain the following formulation:
\begin{subequations}\label{eq:pricing2}
    \begin{align}
        \max\limits_{{X},{Y}} &  \quad \sum_{k \in K} \sum_{i \in V} \sum_{m \in M}\left(p_{mi}Y^k_{mi} \right)&&  \label{eq:pricing2-of}\\
        \text{s.t.} & \quad  \sum_{k\in K} \sum_{m \in M} X^k_{mi} \leq 1 \qquad && \forall\; i \in  V \label{eq:pricing2-partition}\\
        &  \quad {Y}^k \in S^k_{\Psi}({X}^k) && \forall\; k \in K \label{eq:pricing2-valuefunction}\\
        & \quad {X}_m^k \in \{0,1\}^{n}, {Y}_m^k \in \{0,1\}^{n+1} && \forall\;m \in M,k \in K, \label{eq:pricing2-sets}
    \end{align}
\end{subequations}
where $S^k_{\Psi}({X}^k)$ is the set of optimal solutions of the $k$-th follower problem, which, for a given $\tilde{{X}}^k$, reads:
\begin{subequations}\label{eq:pricing2LL}
    \begin{align}
        \Psi^k(\tilde{{X}}^k) = \;  \max\limits_{{Y},{z}} &  \;\sum_{i \in V} \sum_{m \in M} \bar p_{mi} Y^k_{mi}- \sum_{(i,j) \in A} c^k_{ij} z^k_{ij} &&  \label{eq:pricing2LL-of}\\
        \text{s.t.} & \quad  Y^k_{mi} \leq \tilde X^k_{mi} &&\forall\; i \in V, m\in M  \label{eq:pricing2LL-interdiction}\\
        & \quad \sum_{i\in V} \sum_{m \in M} Y^k_{mi} \leq b^k &&  \label{eq:pricing2LL-budget} \\
        & \quad (\sum_{m \in M} {Y}^k_{m},{z}^k)  \text{ is a route} &&  \label{eq:pricing2LL-route}\\
        & \quad {Y}_m^k \in \{0,1\}^{n+1}, {z}^k \in \{0,1\}^{|A|} && \forall\;m \in M.
    \end{align}
\end{subequations}
In this formulation, thanks to the disaggregated variables $Y_{mi}^k$, there is no bilinear product to linearize.
In this case, a route $T \in \mathcal{T}$ corresponds to the pair $(\sum\limits_{m \in M} {Y}^k_{m},{z}^k)$. 
As before, under the assumption that the costs $c^k_{ij}$ satisfy the triangle inequality, the following result holds.
\begin{prop}\label{prop3}
    For any $k\in K$, given a vector $\tilde{{X}}^k \in \{0,1\}^{n}$ satisfying constraints~\eqref{eq:pricing2-partition}, and such that $\sum\limits_{i\in V}\sum\limits_{m \in M} \tilde X^k_{mi}\leq~b^k, \; \forall \; k \in K$, there always exists an optimal solution of the following problem, which is also optimal for $\Psi^k(\tilde{{X}}^k)$:
    \begin{subequations}
        \begin{align}
            \bar \Psi^k(\tilde{{X}}^k) = \; \max\limits_{T \in \mathcal{T}} \left\{\sum_{i \in V(T)}\sum_{m \in M} \bar p_{mi}\tilde X^k_{mi} - C^k(T) \right\}.
        \end{align}
    \end{subequations}
\end{prop}

Note that the proof of Proposition~\ref{prop3} is similar to the one of Proposition~\ref{prop}, so it is omitted.

As for BPFM, we can obtain single-level reformulations of formulation~\eqref{eq:pricing2} through the value function approach or the no-good cuts approach. The value function reformulation of~\eqref{eq:pricing2} reads:

\begin{subequations}\label{eq:pricing2_single}
    \begin{empheq}[left=\hspace*{-1.2cm}(\mathsf{\textcolor{OliveGreen}{BPMD_d}})\empheqlbrace]{align}
        \max\limits_{{X},{Y},{z}} &  \quad \sum_{k \in K} \sum_{i \in V} \sum_{m \in M}p_{mi}Y^k_{mi} &&  \\
        \text{s.t.} &  \;  \sum_{k\in K} \sum_{m \in M} X^k_{mi} \leq 1 && \forall\; i \in  V \label{eq:pricing2_single-partition}\\
        & \;  {\sum_{i\in V} \sum_{m \in M} X^k_{mi} \leq b^k } && \forall\; k \in  K \label{eq:pricing2_single-budget}\\
        &  \;  Y^k_{mi} \leq X^k_{mi} &&\forall\; i \in V, m\in M, k \in K \label{eq:pricing2_single-interdiction}\\
        & \; (\sum_{m \in M} {Y}^k_{m},{z}^k)  \text{ is a route} && \forall\; k \in K \label{eq:pricing2_single-route}\\
        & \hspace*{-4.5mm}\sum_{i \in V} \sum_{m \in M} \bar p_{mi} Y^k_{mi}- \sum_{(i,j) \in A} c^k_{ij} z^k_{ij} \geq && \nonumber\\
        & \qquad \qquad \qquad \sum_{i \in V(T)}\sum_{m \in M} \bar p_{mi}\tilde X^k_{mi} - C^k(T)  && \forall\; T \in \mathcal{T}, k \in K \label{eq:pricing2_single-valuefunction}\\
        & \; {X}_m^k \in \{0,1\}^n, {Y}_m^k \in \{0,1\}^{n+1} && \forall\;m \in M,k \in K\\
        & \; {z}^k \in \{0,1\}^{|A|} && \forall k \in K,
    \end{empheq}
\end{subequations}

where constraints~\eqref{eq:pricing2_single-valuefunction} reformulate, according to Proposition~\ref{prop3}, the value function constraints
\begin{equation*}
    \sum_{i \in V} \sum_{m \in M} \bar p_{mi} Y^k_{mi}- \sum_{(i,j) \in A} c^k_{ij} z^k_{ij} \geq \Psi^k({X}^k) \quad \forall\; k\in K.
\end{equation*}

We can alternatively formulate the BPMD problem through the no-good cuts approach, obtaining formulation $(\mathsf{\textcolor{OliveGreen}{BPMD'}}),$ presented and discussed in Online Appendix~C.

An equivalent single-level formulation, which we call $(\mathsf{\textcolor{OliveGreen}{BPMD_d\text{-}z}})$, can be obtained by projecting out the ${z}$ variables and introducing the new variables $\theta^k$, replacing constraints~\eqref{eq:pricing2_single-route} and~\eqref{eq:pricing2_single-valuefunction} in $(\mathsf{\textcolor{OliveGreen}{BPMD_d}})$ by
\begin{subequations}\label{eq:pricing2-proj}
    \begin{align}
        & \theta^k \geq c^k_{\mathsf{TSP}}(V(T)) \left[\sum_{i \in V(T)} \sum_{m \in M} Y^k_{mi} - |V(T)| + 1\right]  &\forall\;T \in \mathcal{T}, k \in K, \label{eq:pricing2-proj-gamma} \\
        &\sum_{i \in V} \sum_{m \in M} \bar p_{mi}Y^k_{mi}- \theta^k \geq \sum_{i \in V(T)}\sum_{m \in M} \bar p_{mi}\tilde X^k_{mi} - C^k(T) &  \forall\;T \in \mathcal{T}, k \in K, \label{eq:pricing2-proj-valuefunction}
    \end{align}
\end{subequations}
respectively. The obtained formulation $(\mathsf{\textcolor{OliveGreen}{BPMD_d\text{-}z}})$ is reported in Online Appendix~B. 

Furthermore, if a route duration limit has to be considered instead of the capacity constraint, constraints~\eqref{eq:pricing2_single-budget} are replaced by constraints~\eqref{eq:pricing-tmax} for $(\mathsf{\textcolor{OliveGreen}{BPMD_d}})$, or by constraints~\eqref{eq:proj-tmax} for $(\mathsf{\textcolor{OliveGreen}{BPMD_d\text{-}z}})$.

\subsection{Comparing the BPMD formulations}\label{sub:comparison}
In this section, we compare the two BPMD value function formulations, $(\mathsf{\textcolor{OliveGreen}{BPMD}})$ and $(\mathsf{\textcolor{OliveGreen}{BPMD_d}})$, in terms of the value of their linear relaxations. In the following theorem, we prove that the aggregated formulation~$(\mathsf{\textcolor{OliveGreen}{BPMD}})$ is as strong as formulation~$(\mathsf{\textcolor{OliveGreen}{BPMD_d}})$ (i.e., its linear relaxation provides equivalent upper bounds). Let us define $v_{LP}(\mathsf{\textcolor{OliveGreen}{BPMD}})$ and $ v_{LP}(\mathsf{\textcolor{OliveGreen}{BPMD_d}})$ as the optimal values of the LP relaxation of $(\mathsf{\textcolor{OliveGreen}{BPMD}})$ and $(\mathsf{\textcolor{OliveGreen}{BPMD_d}})$, respectively.

\begin{thm}\label{theorem}
Any LP feasible solution $(\tilde{{X}},\tilde{{y}},\tilde{{w}},\tilde{{z}})$ of the model~$(\mathsf{\textcolor{OliveGreen}{BPMD}})$ can be translated into a LP feasible solution $(\tilde{{X}},\tilde{{Y}},\tilde{{z}})$ of the model~$(\mathsf{\textcolor{OliveGreen}{BPMD_d}})$ by imposing that, for all $i \in V, m \in M, k \in K$:
\begin{equation}\label{eq:Yw}
    \tilde Y_{mi}^k =\tilde w^k_{mi}, 
\end{equation}
and vice versa. Hence $v_{LP}(\mathsf{\textcolor{OliveGreen}{BPMD}}) = v_{LP}(\mathsf{\textcolor{OliveGreen}{BPMD_d}})$. 
\end{thm}
\begin{proof}
    Let us start by proving that $v_{LP}(\mathsf{\textcolor{OliveGreen}{BPMD}}) \leq v_{LP}(\mathsf{\textcolor{OliveGreen}{BPMD_d}})$, i.e., that any LP solution $(\tilde{{X}},\tilde{{y}},\tilde{{w}},\tilde{{z}})$ of~$(\mathsf{\textcolor{OliveGreen}{BPMD}})$ is also feasible in $(\mathsf{\textcolor{OliveGreen}{BPMD_d}})$, when imposing Eq.~\eqref{eq:Yw}. Indeed, constraints~\eqref{eq:pricing2_single-partition} and \eqref{eq:pricing2_single-budget} are constraints~\eqref{eq:pricing1-partition} and \eqref{eq:pricing1_single-budget}, respectively. Constraints~\eqref{eq:pricing2_single-interdiction} correspond to \eqref{eq:pricing1-mc2}. Constraints~\eqref{eq:pricing2_single-valuefunction} are equivalent to constraints~\eqref{eq:pricing1_single-valuefunction}. Finally, constraints~\eqref{eq:pricing2_single-route} correspond to \eqref{eq:pricing1_single-route} combined with constraints~\eqref{eq:pricing1-yw}.
    Hence the inequality $v_{LP}(\mathsf{\textcolor{OliveGreen}{BPMD}}) \leq v_{LP}(\mathsf{\textcolor{OliveGreen}{BPMD_d}})$ holds.
    Let us now consider an LP solution $(\tilde{{X}},\tilde{{Y}},\tilde{{z}})$ of~$(\mathsf{\textcolor{OliveGreen}{BPMD_d}})$. We need to show that it is LP feasible also for $(\mathsf{\textcolor{OliveGreen}{BPMD}})$ if Eq.~\eqref{eq:Yw} is imposed. Indeed, as already said, constraints~\eqref{eq:pricing1-partition}, \eqref{eq:pricing1-mc2}, \eqref{eq:pricing1_single-budget}, and \eqref{eq:pricing1_single-valuefunction} correspond to \eqref{eq:pricing2-partition}, \eqref{eq:pricing2_single-interdiction}, \eqref{eq:pricing2_single-budget}, and \eqref{eq:pricing2_single-valuefunction}, respectively. For the remaining constraints, in model~$(\mathsf{\textcolor{OliveGreen}{BPMD_d}})$, we can replace variable $y_{i}^k$ by $\sum\limits_{m \in M} w_{mi}^k$ by constraint~\eqref{eq:pricing1-yw}. Therefore:
    \begin{itemize}
        \item Constraints~\eqref{eq:pricing1-mc3} reads $w_{mi}^k \leq \sum\limits_{u \in M} w_{ui}^k$. They are satisfied by $(\tilde{{X}},\tilde{{Y}},\tilde{{z}})$, because $\tilde Y_{mi}^k \geq 0 \;$ for all $\;m\in M, i \in V, k \in K$, and thus $\tilde Y_{mi}^k \leq \sum\limits_{u \in M} \tilde Y_{ui}^k$.
        \item Constraints~\eqref{eq:pricing1-mc1bis} read $ \sum\limits_{u \in M} w_{ui}^k \leq w_{mi}^k + \sum\limits_{u \in M: u \neq m} X_{ui}^k$. They are satisfied by $(\tilde{{X}},\tilde{{Y}},\tilde{{z}})$ because for all $m \in M, i \in V, k \in K$: 
            \begin{equation*}
               \sum_{u \in M} \tilde Y_{ui}^k \leq \tilde Y_{mi}^k + \sum_{u \in M:u\neq m} \tilde X_{ui}^k \iff  \sum_{u \in M: u \neq m} \tilde Y_{ui}^k \leq \sum_{u \in M:u\neq m} \tilde X_{ui}^k
            \end{equation*}
        hold as $\tilde Y_{ui}^k, \tilde X_{ui}^k \geq 0$ and $\tilde Y_{ui}^k \leq \tilde X_{ui}^k$ for all $u \in M, i \in V, k \in K$ from \eqref{eq:pricing2_single-interdiction}.
        \item Constraints~\eqref{eq:pricing1_single-interdiction} read $\sum\limits_{m \in M} w_{mi}^k \leq \sum\limits_{m \in M} X_{mi}^k$. They are satisfied by $(\tilde{{X}},\tilde{{Y}},\tilde{{z}})$ because $\sum\limits_{m \in M} \tilde Y_{mi}^k \leq \sum\limits_{m \in M} \tilde X_{mi}^k$ holds as $\tilde Y_{mi}^k, \tilde X_{mi}^k \geq 0$ and $\tilde Y_{mi}^k \leq \tilde X_{mi}^k$ for all $m \in M, i \in V, k \in K$ from \eqref{eq:pricing2_single-interdiction}.
        \item Constraints~\eqref{eq:pricing1_single-route} read ``$(\sum\limits_{m \in M} {w}_{m}^k, {z}^k)$ \textit{is a route}'' for all $k \in K$, which correspond to constraints~\eqref{eq:pricing2_single-route}.
    \end{itemize}
    Hence, also inequality $v_{LP}(\mathsf{\textcolor{OliveGreen}{BPMD}}) \geq v_{LP}(\mathsf{\textcolor{OliveGreen}{BPMD_d}})$ holds. 
    Consequently, we have that $v_{LP}(\mathsf{\textcolor{OliveGreen}{BPMD}}) = v_{LP}(\mathsf{\textcolor{OliveGreen}{BPMD_d}}).$ 
\end{proof}

\section{Valid inequalities}\label{sec:valid_inequalities}

In this section, we present some valid inequalities that strengthen the proposed formulations, cutting off parts of the feasible domain because of symmetries or dominance conditions. 

First of all, we consider the setting in which the carriers' problems are all equivalent, i.e., when $b^k$ (or $t_{max}^k$) and ${c}^k$ are the same for all $k$ in $K$. In this case (which is the case we consider in our numerical experiments, see Section~\ref{sec:results}), whenever we impose an inequality for a given $\bar k$, we can add the same inequality for all $k \in K.$ 

In the same setting, we can consider the so-called \textit{symmetry-breaking inequalities} which help reduce the number of equivalent solutions in the platform's feasible region, making it faster for optimization algorithms to find the optimal solution. The symmetry-breaking inequalities read, for formulation~$(\mathsf{\textcolor{OliveGreen}{BPFM}})$:
$$\sum_{i\in V} x_i^{k-1} \geq \sum_{i\in V} x_i^{k} \qquad \forall k \in K\setminus\{1\},$$
and for formulations~$(\mathsf{\textcolor{OliveGreen}{BPMD}})$ and $(\mathsf{\textcolor{OliveGreen}{BPMD_d}})$:
$$\sum_{i\in V}\sum_{m \in M} X_{mi}^{k-1} \geq \sum_{i\in V}\sum_{m \in M} X_{mi}^{k} \qquad \forall k \in K\setminus\{1\},$$
imposing that the number of customers assigned to carrier $k-1$ is greater than the number of customers assigned to carrier $k$.

The second type of valid inequalities we can add is from the family of so-called \textit{cover inequalities}. They can be added to the model also if the carriers are not necessarily equivalent, but have a capacity constraint (i.e., we consider formulations with constraints on the maximum number of packages to serve). This family of inequalities prevents the formation of suboptimal or inefficient customer visit sequences by specifying certain patterns that should be avoided. 
Indeed, for the models with the constraint on the maximum number of customers each carrier $k$ can serve, the platform could identify, for each carrier $k$, all the sets of $b^k$ customers that, even when setting the compensations to the highest values for the carrier (and performing the optimal tour to serve them), would not be served together, because corresponding to a negative profit for carrier $k$. 
Set $S \subset V$ is called a \textit{cover} with respect to carrier $k \in K$, if $|S| = b^k$ and $\sum\limits_{i \in S}(1-m_{\text{min}})p_i - c^k_{\mathsf{TSP}}(S) < 0$. For any $k \in K$, given its cover $S$, we can add the following valid inequality for formulations~$(\mathsf{\textcolor{OliveGreen}{BPFM}})$, and $(\mathsf{\textcolor{OliveGreen}{BPFM\text{-}z}})$:
\begin{equation}\label{eq:cover1}
\sum\limits_{i \in S} x_i^k \leq b^k - 1,
\end{equation}
and for formulations~$(\mathsf{\textcolor{OliveGreen}{BPMD}})$, $(\mathsf{\textcolor{OliveGreen}{BPMD_d}})$, $(\mathsf{\textcolor{OliveGreen}{BPMD\text{-}z}})$, and $(\mathsf{\textcolor{OliveGreen}{BPMD_d\text{-}z}})$:
\begin{equation}\label{eq:cover2}
   \sum\limits_{i \in S} \sum_{m \in M} X_{mi}^k \leq b^k - 1,
\end{equation}
which exclude the identified suboptimal tour.

Identifying all the sets $S$ which correspond to inefficient assignment decisions may be computationally heavy. Indeed, checking for all $k\in K$ all the subsets of cardinality $b^k$ means enumerating $\sum\limits_{k \in K} \binom{n\;}{b^k}$ many subsets~$S$. Furthermore, computing the optimal tour serving the set of customers~$S$ is \NP-hard. Thus, even for small values of $|K|$ and $b^k$, this approach would be computationally intractable. 
Alternatively, one could separate these cover inequalities dynamically and/or heuristically (see Section~\ref{sub:separ}).

Additionally to the above presented valid inequalities, we note that all the formulations could be further strengthened by leveraging the WTA-PTP formulation. Indeed, including constraints imposing the non-negativity of the carriers' profits, i.e., constraints~\eqref{eq:wta-ptp:wta} 
$$\sum_{i \in V} \bar p_{i}^k \alpha_{i}^k - \sum_{(i,j) \in A} c^k_{ij} z_{ij}^k  \geq 0 \qquad \forall\; k \in K$$
for the setting with fixed margins, or constraints~\eqref{eq:wta-ptp-md:wta}
$$\sum_{i \in V} \sum_{m \in M} \bar p_{mi}^k A_{i}^k - \sum_{(i,j) \in A} c^k_{ij} z_{ij}^k  \geq 0 \qquad \forall\; k \in K $$
for the one with margin decisions, leads to better LP relaxations. 
As for the no-good cuts based formulation~$(\mathsf{\textcolor{OliveGreen}{BPFM'}})$ (or $(\mathsf{\textcolor{OliveGreen}{BPMD'}})$ reported in Online Appendix~C), with the same purpose, introducing the binary variable ${z}$, besides constraints~\eqref{eq:wta-ptp:wta} (or \eqref{eq:wta-ptp-md:wta}) one could include also constraints~\eqref{eq:wta-ptp:route} (or \eqref{eq:wta-ptp-md:route}), which impose that ${z}^k$ is the incidence vector of a route that visit the customers $i$ s.t.\ $\alpha_i^k=1$ (or $\sum_{m \in M} A_{mi}^k = 1$, respectively).

\section{Solution approach}\label{sec:solution}

In this section, we present the solution approach we developed to solve the single-level problem reformulations proposed in Section~\ref{sec:formulations}. We designed a branch-and-cut algorithm where value function constraints, the no-good cuts for formulation \eqref{eq:BPFM2}, as well as the ones related to the value of ${\theta}$ when projecting out ${z}$ variables, are separated dynamically as they are exponentially many.
First, in Section~\ref{sub:separ}, we discuss the separation procedures for the exponentially many constraints included in the models.
Second, in Section~\ref{sub:heuristic}, we present a heuristic algorithm, based on the solution of the BPFM, to generate a feasible solution used as a warm-start for the exact approach for the BPMD.

\subsection{Separation procedures}\label{sub:separ}

In this section, we describe the separation procedures we use to dynamically detect violated value function constraints in the single-level reformulations presented in Section~\ref{sec:formulations}.
Indeed, we notice that these constraints are exponentially many and \NP-hard to separate, as they require finding an optimal PTP solution, corresponding to an optimal follower response for a given assignment ${y}$ of the leader. 
As for the formulation with the no-good cuts, these constraints themselves (i.e., constraints~\eqref{eq:BPFM2:nogood} in $(\mathsf{\textcolor{OliveGreen}{BPFM'}})$ and the corresponding ones in $(\mathsf{\textcolor{OliveGreen}{BPMD'}})$) need to be separated by solving a PTP for each carrier. Concerning the formulations with the ${z}$ variables, we notice that there is an exponential number of constraints of type~\eqref{eq:routing_single-valuefunction} in $(\mathsf{\textcolor{OliveGreen}{BPFM}})$, and of type~\eqref{eq:pricing1_single-valuefunction} and~\eqref{eq:pricing2_single-valuefunction} in $(\mathsf{\textcolor{OliveGreen}{BPMD}})$ and $(\mathsf{\textcolor{OliveGreen}{BPMD_d}})$. In the formulations obtained by projecting out the ${z}$ variables, not only constraints~\eqref{eq:routing-proj-cuts} in $(\mathsf{\textcolor{OliveGreen}{BPFM\text{-}z}})$, \eqref{eq:pricing1-proj-valuefunction} in $(\mathsf{\textcolor{OliveGreen}{BPMD\text{-}z}})$, and \eqref{eq:pricing2-proj-valuefunction} in $(\mathsf{\textcolor{OliveGreen}{BPMD_d\text{-}z}})$ are exponential in number, but also constraints~\eqref{eq:nogood-routing} and \eqref{eq:pricing2-proj-gamma}, respectively. The separation of cover inequalities introduced in Section~\ref{sec:valid_inequalities} also requires solving an \NP-hard problem.

We thus propose a separation procedure for these constraints, which we describe in the following referring to the problem with fixed margins, without loss of generality. We note that all constraints are separated on integer solutions only, in order to speed up the solution process.

\textbf{\textit{Separation of constraints~\eqref{eq:routing_single-valuefunction}\\}
}When dealing with formulation~$(\mathsf{\textcolor{OliveGreen}{BPFM}})$, for any given solution $(\tilde{{x}}, \tilde{{y}}, \tilde{{z}})$ of the master problem, obtained by relaxing constraints~\eqref{eq:routing_single-valuefunction}, we solve the PTP with ${x} = \tilde{{x}}$ for each $k$. Let $(\hat{{y}}^k, \hat{{z}}^k)$ be the optimal solution of the latter problem, with $\hat T^k$ and $\hat \varphi^k$  being the corresponding tour and value, respectively. If there exists a $k$ such that
$\sum\limits_{i \in V} \bar p^k_{i}  \tilde y^k_i - \sum\limits_{(i,j) \in A} c^k_{ij} \tilde z^k_{ij} < \hat \varphi^k$, then the following constraint
$$ \sum_{i \in V}  \bar p^k_{i}  y^k_i - \sum_{(i,j) \in A} c^k_{ij} z^k_{ij}  \ge  \sum_{i \in V(\hat T^k)}  \bar p^k_{i} x^k_i - C^k(\hat T^k) $$
is violated by the current solution of the master problem. Thus, we insert it into the master problem. Otherwise, the obtained solution is feasible (optimal if we are at the root node) for the original bilevel formulation.

\textbf{\textit{Separation of constraints~\eqref{eq:nogood-routing}\\}}
When considering formulation~$(\mathsf{\textcolor{OliveGreen}{BPFM\text{-}z}})$, we have to separate constraints~\eqref{eq:nogood-routing} as well, which are also exponentially many. Thus, we first solve the master problem obtained by relaxing both~\eqref{eq:routing-proj-cuts} and \eqref{eq:nogood-routing}, finding the solution $(\tilde{{x}}, \tilde{{y}}, \tilde{{\theta}})$, with the corresponding tour $\tilde T^k$ (with $c^k_{\mathsf{TSP}}(V(\hat T^k))$, the cost of the optimal route associated to $\tilde T^k$ nodes). Then, we solve the lower-level problems with ${x} = \tilde{{x}}$ for each $k$, obtaining the optimal solution $(\hat{{y}}^k, \hat{{z}}^k)$, the corresponding tour $\hat T^k$, and the value $\hat \varphi^k$. 

In the case in which there exists a $k$ such that $\tilde \theta^k  < c^k_{\mathsf{TSP}}(V(\tilde T^k))$, we insert the following constraint:
\begin{equation}\label{eq:separation_theta}
    \theta^k  \geq c^k_{\mathsf{TSP}}(V(\tilde T^k)) \left[\sum_{i \in V(\tilde T^k)} y^k_i - |V(\tilde T^k)| + 1\right].
\end{equation}

If instead, for all $k$, inequality
\begin{equation} \label{eq:check}
    \tilde \theta^k  \geq c^k_{\mathsf{TSP}}(V(\tilde T^k)),
\end{equation}
does hold, then we proceed with the separation of \eqref{eq:routing-proj-cuts} as we did for constraints~\eqref{eq:routing_single-valuefunction} above, i.e., if there exists a $k$ such that
$\sum\limits_{i \in V} \bar p^k_{i}  \tilde y^k_i - \tilde \theta^k < \hat \varphi^k$, then the we add the following constraint
$$\sum_{i \in V}  \bar p^k_{i}  y^k_i - \theta^k  \ge  \sum_{i \in V(\hat T^k)}  \bar p^k_{i} x^k_i - C^k(\hat T^k) $$
to the master. 

\textbf{\textit{Separation of constraints~\eqref{eq:proj-tmax}\\}}
When taking into account the maximum duration constraints~\eqref{eq:proj-tmax} on variable $\theta^k$, we use the following separation procedure for each $k$:
\begin{itemize}
    \item if $c^k_{\mathsf{TSP}}(V(\tilde T^k)) \leq t_{max}^k$, we add cut~\eqref{eq:separation_theta};
    \item otherwise (i.e., if $c^k_{\mathsf{TSP}}(V(\tilde T^k)) > t_{max}^k$), we cut off the solution $\tilde{{y}}^k$ because it is infeasible, adding the following no-good cut: $\sum\limits_{i \in V(\tilde T^k)} y_i^k \leq |V(\tilde T^k)| -1$.
\end{itemize}

\textbf{\textit{Separation of no-good cuts~\eqref{eq:BPFM2:nogood}\\}}
Given formulation~$(\mathsf{\textcolor{OliveGreen}{BPFM'}})$, the separation procedure we follow consists in: $a)$ solving the relaxation obtained by relaxing the no-good cuts~\eqref{eq:BPFM2:nogood} to get a solution $\tilde{{\alpha}}$ and the corresponding sets $\tilde{\mathcal{V}}^k =\{i \in V\,: \,\alpha^k_i = 1 \}$ for all $k\in K$; $b)$ solving the $|K|$ PTP problems~\eqref{eq:routingLL} with ${x} = {\alpha}$ obtaining the optimal carriers' responses; and, $c)$ in case there is a $k\in K$ for which ${y}^k < {\alpha}^k$, the following no-good cut is added to the relaxation of~\eqref{eq:BPFM2}
$$\sum_{i \in \tilde{\mathcal{V}}^k}  (1 - \alpha^k_i) +  \sum_{i \not \in  \tilde{\mathcal{V}}^k} \alpha^k_{i} \geq 1, $$
otherwise, the obtained solution is feasible for the original formulation. 

We highlight here that this type of cut only excludes the current tentative assignment of the master problem, which explains why the use of these no-good cuts turns out to be inefficient in practice (see numerical results Section~\ref{sec:results}).

\textbf{\textit{Separation of cover inequalities~\eqref{eq:cover1} and \eqref{eq:cover2}\\}}
In Section~\ref{sec:valid_inequalities}, cover inequalities have been introduced as valid inequalities for formulations with the bound on the number of packages each carrier can deliver. They can be separated dynamically in the following way. For a given $k\in K$, let $\tilde{{x}}^k$ be a solution to the master problem. If $\sum_{i \in V} \tilde x_i^k = b^k$, let $\tilde S$ be the support set of $\tilde{{x}}^k.$ If $\sum\limits_{i \in \tilde S}(1-m_{\text{min}})p_i - c^k_{\mathsf{TSP}}(\tilde S) < 0,$ then $\tilde S$ defines a cover, and we can add the corresponding cover inequality~\eqref{eq:cover1} or \eqref{eq:cover2} to the model.

As an alternative, the following heuristic can be executed as a preprocessing step in order to find some of the covers $S$. For each $k$, sort the customers in non-decreasing order of the associated compensation $(1-m_{\text{min}})p_i$; take the first $\lceil b^k / 2 \rceil$ customers according to the sorting, defining the set $\bar S$; iteratively, for each vertex $u \in \bar S$, select the $\lfloor b^k / 2 \rfloor$ customers whose distances from $u$ are maximal. In this way, one obtains $\lceil b^k / 2 \rceil$ sets $S$ for which a valid cover inequality can be added.

Finally, one can note that if $$\min\limits_{\substack{S \subset V \\|S|=b^k }} \left\{\sum\limits_{i \in S}(1-m_{\text{min}})p_i - C^k(S) \right\}\geq 0$$ for a given $k$, no violated cover inequality exists for the corresponding carrier $k$. We perform this check as a preprocessing step in our experiments.

\subsection{Heuristic solution approach}\label{sub:heuristic}

In this section, we describe a heuristic procedure used to obtain a warm-start feasible solution for the branch-and-cut algorithm solving the BPMD, which possibly corresponds to a tighter lower bound of the problem compared to the one provided by the $(\mathsf{\textcolor{OliveGreen}{UCC\text{-}PTP}})$, or the one found by solving the $(\mathsf{\textcolor{OliveGreen}{WTA\text{-}PTP}})$ and retrieving its associated feasible solution. It consists of three phases, involving the solution of the BPFM. 

First, the BPFM (modeled either as $(\mathsf{\textcolor{OliveGreen}{BPFM}})$, or as $(\mathsf{\textcolor{OliveGreen}{BPFM\text{-}z}})$, or as $(\mathsf{\textcolor{OliveGreen}{BPFM'}})$) is solved by setting all the compensations to their highest values, i.e., $\bar p_i^k=(1-m_{\text{min}})p_i$ for all $i\in V, k \in K$. In this way, we obtain a first feasible solution to the BPMD. Note that, in case no package is assigned in this solution, then it means that no other solution exists in which any package is served. Indeed, this is the most rewarding compensation decision for the followers, and assigning no packages means that the compensations are still too low with respect to the routing costs. 

In case the solution to the first step is nonempty, in the second step, BPMD is solved with all the variables, but the margin variables, fixed to the values returned by the first phase, plus an additional constraint imposing that the profit of each carrier should be nonnegative. This corresponds to solving the following auxiliary problem, {which is a multiple-choice knapsack solved independently for each follower}:
\begin{subequations}\label{eq:single}
    \begin{align}
        \max_{X} & \sum\limits_{k \in K} \sum\limits_{m \in M} \sum\limits_{i \in V(\hat T^k)} p_{mi}X_{mi}^k && \\
        \text{s.t.} & \sum\limits_{m \in M} X_{mi}^k = \hat x^k_i && \forall\; i \in V,k \in K \\
        & \sum_{m \in M}\sum_{i \in V(\hat T^k)} p_{mi}X_{mi}^k \leq \sum_{i \in V(\hat T^k)}p_i - C^k(\hat T^k) && \forall\;k\in K,
    \end{align}
\end{subequations}

\noindent where $\hat{{x}}$, and $\hat T$ come from the assignment and routing decisions made at the first step of the heuristic. 
The solution to the above problem determines which is the best margin decision for the leader, i.e., $\check{{X}}$, given the assignment of the items decided in the previous phase.

Finally, in the third step, the BPFM is solved again by fixing the margins according to the solution obtained in the second step, i.e., for each $i$ and $k$, we set $\bar p_i^k=p_i-\sum\limits_{m \in M}\check{X}_{mi}^kp_{mi}.$
The solution obtained in the third step, together with the $\check{{X}}$ found in the second step, is a feasible solution to the problem. Note that this last step is needed as, otherwise, the solution obtained from the second step might not be bilevel feasible, as it does not necessarily optimize the followers' objective. The algorithm will then return the best solution in terms of the leader's objective function between the one found in the first and the third step. 

As already noted, the separation of no-good cuts is computationally more expensive than that of value function cuts (see Section~\ref{sub:fixed}), thus the employment of $(\mathsf{\textcolor{OliveGreen}{BPFM'}})$ within the heuristic is not preferable.
Depending on the presence of a budget constraint (\eqref{eq:routingLL-budget}, \eqref{eq:pricing1LL-budget} or \eqref{eq:pricing2LL-budget}) or a route duration constraint (\eqref{eq:pricing-tmax} or \eqref{eq:proj-tmax}) the problem to solve at each step changes. Indeed, when there is a constraint on the duration of the route, it is better to have a problem with the ${z}$ variables in the upper level, i.e., $(\mathsf{\textcolor{OliveGreen}{BPFM}})$, because constraints~\eqref{eq:pricing-tmax} do not need to be separated. Instead, some preliminary experiments showed that, as expected, when there is no route duration constraint, the problem without the ${z}$ variables $(\mathsf{\textcolor{OliveGreen}{BPFM\text{-}z}})$ turns out to be faster.
The overall heuristic is described in Algorithm~\ref{algo:heuristic1}.

\begin{algorithm}[ht]\caption{Heuristic algorithm for BPMD}
    \LinesNumbered
    \SetAlgoLined 
    \label{algo:heuristic1}
    \KwData{Sets $V, A, K, M=\{m_{\text{min}},\dots,m_{\text{max}}\},$ and parameters $p_i$ for all $i \in V$, $c^k_{ij}$ for all $(i,j) \in A$, and $b^k$ (or $t_{max}^k$) for all $k \in K$.} 
    \KwResult{Feasible solution for BPMD.} 
    Set $\bar p_i^k=(1-m_{\text{min}})p_i$ for all $i\in V, k \in K.$ Solve BPFM. Return the optimal solution in terms of assignment of the leader $\hat{{x}}$, and acceptance and routing decisions of the followers $\hat{T}$.\\
    Solve problem~\eqref{eq:single}. Return the optimal solution $\check{{X}}$. \\
    Set $\bar p_i^k=p_i-\sum\limits_{m \in M}\check{X}_{mi}^kp_{mi}$ for all $i\in V, k \in K.$ Solve BPFM. Return the optimal solution $(\check{{x}}, \check{T})$.\\
    \Return the best solution, in terms of the leader's objective function, between $(\hat{x}, \hat{T})$ and $(\check{x}, \check{T})$.
    \end{algorithm}

\section{Numerical Results}\label{sec:results}

We test the different models developed in the previous sections on two sets of instances.

The first set of 15 instances is generated in the following way: the graph $\mathcal{G}=(V, A)$ and the number of vehicles $|K|$ are taken from the ``\textit{Chao instances}'' \citep{chao1996}, originally proposed for the team orienteering problem (the number of customers $|V|$ is in $\{20,31,32,63,65\}$ and the number of vehicles $|K|$ is in $\{2,3,4\}$); costs $c^k_{ij}$ are assumed to be the same for all the vehicles $k \in K$ and equal to the Euclidean distance between $i$ and $j$ for each arc $(i,j)$, rounded to the nearest not smaller integer; node prices $p_i$ are generated pseudo-randomly in $[0,100]$ following the Generation 2 procedure proposed in \cite[Section 6]{fischetti1998}, consisting in setting 
$p_i:=1+(7141\cdot i+73)\text{ mod} (100)$ for each $i\in V$. The prices are available as a feature in the instances proposed in \cite{chao1996}, but these instances are conceived for team orienteering problem applications with two depots (one is the starting point and the other is the ending point of the vehicles' routes), thus, as we are in a one-depot framework, we generate the nodes prices according to the previously presented formula. We recall that each customer is associated with the demand for a single item.

The second set of 12 instances has the following characteristics: the graphs $\mathcal{G}=(V, A)$ are taken from the famous benchmark set of instances known as ``\textit{Solomon instances}'' R101 -- randomly distributed customers -- \citep{solomon}, with a number of customers in $\{20,25,30,35\}$; we set $|K|\in\{1,2,3\}$ following the same setting of the \textit{Chao instances}; costs $c^k_{ij}$ are again assumed to be the same for all the vehicles $k \in K$ and equal to the Euclidean distances; node prices $p_i$ are generated as for the former set of instances.

The bound $b^k$ on the number of items each vehicle can serve is set to $\left \lceil\frac{|V|}{|K|}\right \rceil +2$. The upper bound on the duration of the route $t_{max}^k$, which we discussed in Section~\ref{sub:tmax} is only available for the \textit{Chao instances}, thus we only solve this type of instance when considering the duration constraint. In particular, each \textit{Chao instance} has a specific value of $t_{max}^k$ identified by an alphabetic letter. We take the instances of type ``k''.

The proposed formulations are implemented in Python 3.10 and solved by using the Cplex solver (version 22.1.0.0) \citep{cplex}, with a time limit of one hour. All the experiments are conducted on a 3.7 GHz Intel Xeon W-2255 CPU, 128 GB RAM. 

We present the numerical results obtained by testing the formulations of the problem with fixed margin in Section~\ref{sub:fixed}, and the ones obtained by testing the formulations of the problem with margin decisions in Section~\ref{sub:marg}. In Section~\ref{sub:gain}, we discuss the gain of the margin decisions through the analysis of two \textit{Solomon instances}.

\subsection{Results on the fixed margins problem}\label{sub:fixed}
In this section, we report the summary of the results obtained by solving the two formulations proposed for the problem with fixed margins: $(\mathsf{\textcolor{OliveGreen}{BPMF}})$, and $(\mathsf{\textcolor{OliveGreen}{BPMF'}})$. We test these models on the \textit{Chao} and the \textit{Solomon instances} described above, considering margins: $0.2, 0.5, 0.7, 0.8$, and $0.9$. 
For each instance, we assume the same margin $\hat m$ is applied to all items and carriers, i.e., $\bar p_i^k = (1-\hat m)p_i$ for all $i \in V, k \in K$. 
In Table~\ref{tab:fixed}, there are two blocks of columns, one for each of the formulations considered. For each of the two we report: \textit{LB} and	\textit{UB}, the average lower and upper bound at termination, respectively; \textit{gap}, the average percentage gap returned by Cplex at termination; \textit{time}, the average computing time in seconds; \textit{\#nodes}, the average number of nodes of the branch-and-cut tree at termination; \textit{\%served}, the percentage of served costumers. The first part of the table is associated with the 15 \textit{Chao instances}, and the second part with the 12 \textit{Solomon instances}. 

\begin{table}[ht!]
    \centering
    \scalebox{0.6}{\begin{tabular}{|l|r r r r r r |r r r r r r|} \hline
        \multirow{2}{*}{} &  \multicolumn{6}{c|}{Model~$(\mathsf{\textcolor{OliveGreen}{BPFM}})$} & \multicolumn{6}{c|}{Model~$(\mathsf{\textcolor{OliveGreen}{BPFM'}})$} \\ \cline{2-13}
        & LB & UB & gap & time & \#nodes & \%served & LB & UB & gap & time & \#nodes &   \%served  \\ \hline
        \textit{Chao instances} & \multicolumn{12}{c|}{} \\ \hline
        \cellcolor[HTML]{D9E1F2}0.2 & 430.4 & 430.4 & {0.00} & {562.7} & 37 & 100 & 394.2 & 430.4 & 8.41 & 1292.7 & 48 & 94.1 \\ \hline
        \cellcolor[HTML]{D9E1F2}0.5 & 1076.0 & 1076.0 & 0.00 & {260.8} & 120 & 100 & 1076.0 & 1076.0 & 0.00 & 836.6 & 112 & 99.9 \\ \hline
        \cellcolor[HTML]{D9E1F2}0.7 & 1377.6 & 1506.7 & {8.55} & {1516.5} & 2060 & 92.9 & 1230.1 & 1506.4 & 18.34 & 1572.9 & 800 & 83.6 \\ \hline
        \cellcolor[HTML]{D9E1F2}0.8 & 1631.2 & 1721.6 & {5.25} & {2876.4} & 16745 & 92.7 & 1074.0 & 1721.6 & 37.61 & 2990.0 & 3180 & 64.5 \\ \hline
        \cellcolor[HTML]{D9E1F2}0.9 & 1601.4 & 1936.8 & {17.32} & 3600 & 202061 & 76.5 & 831.4 & 1936.8 & 57.07 & 3600 & 259301 & 43.1 \\ \hline\hline
        \textit{Solomon instances} & \multicolumn{12}{c|}{} \\ \hline
        \cellcolor[HTML]{D9E1F2}0.2 & 270.5 & 270.5 & {0.00} & {256.5} & 1707 & 100 & 270.3 & 270.5 & 0.08 & 538.9 & 694 & 99.5 \\ \hline
        \cellcolor[HTML]{D9E1F2}0.5 & 674.6 & 676.3 & {0.24} & {1307} & 39087 & 98.3 & 620.2 & 676.3 & 8.29 & 2034 & 8682 & 90.5 \\ \hline
        \cellcolor[HTML]{D9E1F2}0.7 & 779.7 & 942.1 & 17.2 & 2564 & 420935 & 74.1 & 774.7 & 943.4 & 17.9 & 2742 & 725464 & 73.5 \\ \hline
        \cellcolor[HTML]{D9E1F2}0.8 & 357.3 & 357.3 & {0.00} & {205.3} & 57739 & 20.4 & 357.3 & 362.6 & 1.43 & 498.1 & 102156 & 21.0 \\ \hline
        \cellcolor[HTML]{D9E1F2}0.9 & 0.0 & 0.0 & {0.00} & 0.3 & 0 & 0 & 0.0 & 0.0 & 0.00 & 0.3 & 0 & 0 \\ \hline
    \end{tabular}}
    \caption{Comparison between model~$(\mathsf{\textcolor{OliveGreen}{BPFM}})$ and $(\mathsf{\textcolor{OliveGreen}{BPFM'}})$.}
    \label{tab:fixed}
\end{table}
It is clear from these results that $(\mathsf{\textcolor{OliveGreen}{BPFM}})$ is the best formulation in terms of computational efficiency. For all the tested margins, indeed, both the gap and the time of $(\mathsf{\textcolor{OliveGreen}{BPFM}})$ are better than the ones of $(\mathsf{\textcolor{OliveGreen}{BPFM'}})$. These results motivate the use of model  $(\mathsf{\textcolor{OliveGreen}{BPFM}})$ in the execution of steps 1 and 3 of the heuristic given in Algorithm~\ref{algo:heuristic1}, as well as the choice of discarding $(\mathsf{\textcolor{OliveGreen}{BPMD'}})$ in the tests presented in the following section.

\subsection{Results on the problem with margin decisions}\label{sub:marg}
In this section, we discuss the numerical results obtained by testing the formulations proposed for the problem with margin decisions. We consider the following sets of margins $M$: $\{0.2, 0.5\},$ $\{0.5, 0.9\},$ $\{0.2, 0.5, 0.8\},$ and $\{0.5, 0.7, 0.9\}.$ We restrict our tests to two or three margins as one might reasonably assume that, in practical settings, the choice may be among low and high margins or low, medium, and high margins. The value of the margins is selected after some preliminary experiments aimed at identifying values that generate different solution structures (as shown in the following).

The feasible solutions found by the heuristic Algorithm~\ref{algo:heuristic1} are used as MIP start for Cplex. The heuristic solves at steps 1 and 3 either the $(\mathsf{\textcolor{OliveGreen}{BPFM\text{-}z}})$ formulation or the $(\mathsf{\textcolor{OliveGreen}{BPFM}})$ formulation when the budget constraint or the route duration limit are considered, respectively. 
A time limit of 1 hour is set for each heuristic phase and for the models solution as well. 

A first set of experiments is performed comparing the bilevel solutions obtained through the heuristic proposed in Algorithm~\ref{algo:heuristic1}, with the ones of the $(\mathsf{\textcolor{OliveGreen}{WTA\text{-}PTP\text{-}MD}})$ formulation~\eqref{eq:wta-ptp-md} with margin decisions and the UCC-PTP formulation~\eqref{eq:ucc-ptp} for the lowest values of margins in $M$ (i.e., the highest compensation for the carriers). Indeed, we cannot include margin decisions in the UCC-PTP problem, as it is reasonable to assume that, in this setting, the carriers will always select the highest possible compensation if multiple ones are available, which means setting $\bar{p}_i^k=(1-m_{min})p_i \; \forall i\in V, k\in K$ in the objective function of formulation~\eqref{eq:ucc-ptp}.
We perform these experiments first of all in order to evaluate the quality of the feasible solutions returned by the heuristic proposed in Algorithm~\ref{algo:heuristic1}, which solves a sequence of BPFM formulations, against the quality of the solutions of the single-level UCC-PTP and of the \textit{recovered solutions} of the single-level $(\mathsf{\textcolor{OliveGreen}{WTA\text{-}PTP\text{-}MD}})$. Furthermore, these tests allow us to evaluate the impact of changes in the problem setting on the value of the platform solution. 

Tables~\ref{tab:h1} and \ref{tab:h2} summarize the results on the instances with capacity constraints and route duration constraints, respectively. 
For the heuristic Algorithm~\ref{algo:heuristic1}, and the UCC-PTP formulation~\eqref{eq:ucc-ptp}, we report: \textit{LB}, the average lower bound returned by the method; \textit{time}, the average computing time (in seconds) needed to return the solution; \textit{gap\_LB}, the average percentage gap between \textit{LB} and the best lower bound \textit{LB}$^*$ on the optimal solution returned by the four methods~$(\mathsf{\textcolor{OliveGreen}{BPMD}})$, $(\mathsf{\textcolor{OliveGreen}{BPMD_d}})$, $(\mathsf{\textcolor{OliveGreen}{BPMD\text{-}z}})$ and $(\mathsf{\textcolor{OliveGreen}{BPMD_d\text{-}z}})$ (\textit{gap\_LB}$=100\frac{\text{\textit{LB}}-\text{\textit{LB}$^*$}}{\text{\textit{LB}$^*$}}$). 
For the $(\mathsf{\textcolor{OliveGreen}{WTA\text{-}PTP\text{-}MD}})$ formulation~\eqref{eq:wta-ptp-md}, we report: \textit{LBrec}, the average profit gained by the platform when considering carriers' reaction to the assignment corresponding to the $(\mathsf{\textcolor{OliveGreen}{WTA\text{-}PTP\text{-}MD}})$ solution (i.e., the value of what we define as \textit{recovered solution}, see Section~\ref{sub:singlelev}); \textit{UB}, the average upper bound returned by the method; \textit{time}, the average computing time needed to return the corresponding solution; \textit{gap\_UB}, the average percentage gap between \textit{UB} and the best lower bound \textit{LB}$^*$ on the optimal solution returned by the four methods~$(\mathsf{\textcolor{OliveGreen}{BPMD}})$, $(\mathsf{\textcolor{OliveGreen}{BPMD_d}})$, $(\mathsf{\textcolor{OliveGreen}{BPMD\text{-}z}})$ and $(\mathsf{\textcolor{OliveGreen}{BPMD_d\text{-}z}})$ (\textit{gap\_UB}$=100\frac{\text{\textit{UB}}-\text{\textit{LB}$^*$}}{\text{\textit{LB}$^*$}}$); \textit{gap\_LBrec}, the average percentage gap between the profit associated with the \textit{recovered solution}, i.e., \textit{LBrec}, and \textit{LB}$^*$  (\textit{gap\_LBrec}$=100\frac{\text{\textit{LBrec}}-\text{\textit{LB}$^*$}}{\text{\textit{LB}$^*$}}$).

\begin{table}[ht!]
    \centering
    \scalebox{0.7}{\begin{tabular}{|l|rrr|rrr|rrrrr|} \hline
        \multirow{2}{*}{} & \multicolumn{3}{c|}{Heuristic} &  \multicolumn{3}{c|}{$(\mathsf{\textcolor{OliveGreen}{UCC\text{-}PTP}})$}  & \multicolumn{5}{c|}{$(\mathsf{\textcolor{OliveGreen}{WTA\text{-}PTP\text{-}MD}})$} \\ \cline{2-12}
        & LB & time & gap\_LB 
        & LB & time & gap\_LB 
        & LBrec & UB & time & gap\_UB & {gap\_LBrec}\\ \hline
        \textit{Chao instances} & \multicolumn{11}{c|}{} \\ \hline
        \cellcolor[HTML]{D9E1F2}\{0.2, 0.5\} & 1076 & 141 & 0.00 & 430 & 1759 & -60.1 & 1072 & 1076 & 0.48 & 0.00 & -0.45 \\ \hline
        \cellcolor[HTML]{D9E1F2}\{0.5, 0.9\} & 1876 & 3484 & -1.88 & 1074 & 1782 & -43.8 & 1730 & 1937 & 243 & 1.54 & -13.6 \\ \hline
        \cellcolor[HTML]{D9E1F2}\{0.2, 0.5, 0.8\} & 1716 & 2344 & -0.28 & 430 & 1759 & -75.0 & 1664 & 1722 & 0.91 & 0.18 & -4.58 \\ \hline
        \cellcolor[HTML]{D9E1F2}\{0.5, 0.7, 0.9\} & 1881 & 3336 & -2.02 & 1074 & 1782 & -43.9 & 1734 & 1937 & 16.5 & 1.19 & -10.5 \\  \hline\hline
        \textit{Solomon instances} & \multicolumn{11}{c|}{} \\ \hline  
        \cellcolor[HTML]{D9E1F2}\{0.2, 0.5\} & 672 & 1865 & -0.63 & 266 & 330.4 & -60.6 & 624 & 676 & 0.45 & 0.08 & -7.97 \\ \hline
        \cellcolor[HTML]{D9E1F2}\{0.5, 0.9\} & 770 & 4902 & -15.4 & 656 & 315.9 & -28.0 & 576 & 1006 & 1478 & 9.4 & -36.3 \\ \hline
        \cellcolor[HTML]{D9E1F2}\{0.2, 0.5, 0.8\} & 790 & 3927 & -16.8 & 266 & 330.4 & -71.9 & 727 & 1007 & 1119 & 5.36 & -23.7 \\ \hline
        \cellcolor[HTML]{D9E1F2}\{0.5, 0.7, 0.9\} & 773 & 4197 & -18.1 & 656 & 315.9 & -30.9 & 628 & 1014 & 1544 & 6.28 & -32.0 \\ \hline
    \end{tabular}}
    \caption{Comparison between Algorithm~\ref{algo:heuristic1}, the single-level models $(\mathsf{\textcolor{OliveGreen}{UCC\text{-}PTP}})$ and $(\mathsf{\textcolor{OliveGreen}{WTA\text{-}PTP\text{-}MD}})$.}
    \label{tab:h1}
\end{table}

\begin{table}[ht!]
    \centering
    \scalebox{0.7}{\begin{tabular}{|l|rrr|rrr|rrrrr|} \hline
        \multirow{2}{*}{\textit{Chao instances}} & \multicolumn{3}{c|}{Heuristic} &  \multicolumn{3}{c|}{$(\mathsf{\textcolor{OliveGreen}{UCC\text{-}PTP}})$}  & \multicolumn{5}{c|}{$(\mathsf{\textcolor{OliveGreen}{WTA\text{-}PTP\text{-}MD}})$} \\ \cline{2-12}
        & LB & time & gap\_LB 
        & LB & time & gap\_LB
        & LBrec & UB & time & gap\_UB & gap\_LBrec \\ \hline
        \cellcolor[HTML]{D9E1F2}\{0.2, 0.5\} & 529 & 3277 & -0.04 & 211 & 1783 & -60.1 & 526 & 585 & 2014 & 9.22 & -0.47\\ \hline
        \cellcolor[HTML]{D9E1F2}\{0.5, 0.9\} & 941 & 3765 & -1.77 & 529 & 1790 & -44.1 & 920 & 1056 & 1983 & 10.1 & -4.60 \\ \hline
        \cellcolor[HTML]{D9E1F2}\{0.2, 0.5, 0.8\} & 841 & 3665 & -0.74 & 211 & 1783 & -75.1 & 839 & 944 & 2008 & 9.82 & -0.90 \\ \hline
        \cellcolor[HTML]{D9E1F2}\{0.5, 0.7, 0.9\} & 942 & 3573 & -1.55 & 529 & 1790 & -44.3 & 922 & 1057 & 2017 & 9.78 & -4.64 \\  \hline
    \end{tabular}}
    \caption{Comparison between Algorithm~\ref{algo:heuristic1}, the single-level models $(\mathsf{\textcolor{OliveGreen}{UCC\text{-}PTP}})$ and $(\mathsf{\textcolor{OliveGreen}{WTA\text{-}PTP\text{-}MD}})$, when considering the route duration limit.}
    \label{tab:h2}
\end{table}
As expected, we observe that the average gaps with respect to the bilevel solutions returned by the heuristic are much tighter than the ones obtained by solving the $(\mathsf{\textcolor{OliveGreen}{UCC\text{-}PTP}})$ formulation with the highest compensations. When instead the $(\mathsf{\textcolor{OliveGreen}{WTA\text{-}PTP\text{-}MD}})$ model is solved, we assume that the carriers can only either accept or reject the whole bundle of assigned items. This reduced decision power of the carriers with respect to the bilevel setting leads to an overestimation of the platform profit, as shown by the fact that the gap\_UB values are always positive, 
whereas the absolute value of the gap\_LBrec is always higher than the one of the gap\_LB returned by the heuristic. Indeed, although the UB gaps returned by the $(\mathsf{\textcolor{OliveGreen}{WTA\text{-}PTP\text{-}MD}})$ formulation are tight, the solutions that correspond to them are not \textit{bilevel feasible} in our setting, i.e., when the assignment is given to the carriers and they solve their own PTP models, some of the items are not accepted and the true profit of the platform is lower than what expected (with a gap w.r.t. LB$^*$ of up to -63\%).

At this point, we move to discuss the results of the BPMD models. 
The results on  models~$(\mathsf{\textcolor{OliveGreen}{BPMD}})$ and $(\mathsf{\textcolor{OliveGreen}{BPMD_d}})$, and their respective versions without ${z}$ variables introduced in Section~\ref{sub:proj}, $(\mathsf{\textcolor{OliveGreen}{BPMD\text{-}z}})$ and $(\mathsf{\textcolor{OliveGreen}{BPMD_d\text{-}z}})$, either with the constraints on the number of packages or the duration of the route, are reported Tables~\ref{tab:1}--\ref{tab:2} and Tables~\ref{tab:3}--\ref{tab:4}, respectively. Tables~\ref{tab:1} and~\ref{tab:3} have two blocks of columns: one for model~$(\mathsf{\textcolor{OliveGreen}{BPMD}})$ and the other one for $(\mathsf{\textcolor{OliveGreen}{BPMD\text{-}z}})$. Similarly, Tables~\ref{tab:2} and~\ref{tab:4} have two blocks of columns: one for model~$(\mathsf{\textcolor{OliveGreen}{BPMD_d}})$ and the other one for $(\mathsf{\textcolor{OliveGreen}{BPMD_d\text{-}z}})$. For each model, we report: \textit{$\text{LB}_h$}, the average lower bound returned by the heuristic given in Algorithm~\ref{algo:heuristic1}; \textit{\#opt}, the number of instances solved to optimality; \textit{LB}, the average lower bound at termination; \textit{UB}, the average upper bound at termination; \textit{gap}, the average percentage gap returned by Cplex at termination; $\overline{\text{\textit{gap}}}$, the average percentage gap between the $\text{\textit{UB}}^{\min}$ (minimum between \textit{UB} and the upper bound returned by the solution of the $(\mathsf{\textcolor{OliveGreen}{WTA\text{-}PTP\text{-}MD}})$ formulation~\eqref{eq:wta-ptp-md}) and LB; \textit{time}, the average computing time in seconds; \textit{septime}, the average time in seconds needed for the separation of both the value function constraints, and either the subtour constraints or the route value constraint (when ${z}$ variables are projected out); \textit{\#sep}, the average number of separated integer solutions; \textit{\#nodes}, the average number of nodes of the branch-and-cut tree at termination. 
Tables~\ref{tab:1} and \ref{tab:2} consist of two parts: the first part is associated with the 15 \textit{Chao instances}, and the second part with the 12 \textit{Solomon instances}. Instead, Tables~\ref{tab:3} and \ref{tab:4} report average results on the \textit{Chao instances} only because they are the only ones with the information on the route duration. 
Each row reports the margin set $M$ considered.

\begin{table}[hb]
    \centering
    \scalebox{0.6}{\begin{tabular}{|l|r|r r r r r r r r | r |r r r r r r r r | r |} \hline
        \multirow{2}{*}{} & {Heuristic} &  \multicolumn{9}{c|}{Model~$(\mathsf{\textcolor{OliveGreen}{BPMD}})$} & \multicolumn{9}{c|}{Model~$(\mathsf{\textcolor{OliveGreen}{BPMD\text{-}z}})$} \\ \cline{2-20}
        & $\text{LB}_h$
        & \#opt & LB & UB & gap & time & septime & \#sep & \#nodes & $\overline{\text{gap}}$ &
        \#opt & LB & UB & gap & time & septime & \#sep & \#nodes  & $\overline{\text{gap}}$  \\ \hline
        \textit{Chao instances} & \multicolumn{19}{c|}{} \\ \hline
        \cellcolor[HTML]{D9E1F2}\{0.2, 0.5\} & 1076 & 15 & 1076 & 1076 & 0.00 & 6.3 & 0.0 & 1 & 0 & 0.00 & 15 & 1076 & 1076 & 0.00 & 6.8 & 1.6 & 3 & 0 & 0.00 \\ \hline
        \cellcolor[HTML]{D9E1F2}\{0.5, 0.9\} & 1876 & 0 & 1892 & 1937 & 2.36 & 3600 & 1254 & 384 & 193346 & 2.47 & 0 & 1907 & 1937 & 1.60 & 3600 & 2541 & 4472 & 121306 & 1.60 \\ \hline
        \cellcolor[HTML]{D9E1F2}\{0.2, 0.5, 0.8\} & 1716 & 11 & 1719 & 1722 & 0.18 & 1477 & 723 & 309 & 36777 & 0.18 & {11} & 1719 & 1722 & 0.17 & 1337 & 1037 & 2337 & 19481 & 0.17\\ \hline
        \cellcolor[HTML]{D9E1F2}\{0.5, 0.7, 0.9\} & 1881 & 0 & 1893 & 1937 & 2.33 & 3600 & 1310 & 430 & 170169 & 2.33 & 0 & 1911 & 1937 & 1.34 & 3600 & 2397 & 4844 & 143357 & 1.34 \\ \hline\hline
        \textit{Solomon instances} & \multicolumn{19}{c|}{} \\ \hline
        \cellcolor[HTML]{D9E1F2}\{0.2, 0.5\} & 672 & 9 & 675 & 676 & 0.20 & 904 & 114 & 98 & 50971 & 0.20 & 9 & 676 & 676 & 0.08 & 986 & 827 & 1409 & 18437 & 0.08 \\\hline
        \cellcolor[HTML]{D9E1F2}\{0.5, 0.9\} & 770 & {2} & 898 & 1003 & {9.04} & 3163 & 245 & 377 & 579250 & 8.76 & 0 & 875 & 1068 & 17.0 & 3600 & 980 & 6279 & 435079 & 12.0 \\\hline
        \cellcolor[HTML]{D9E1F2}\{0.2, 0.5, 0.8\} & 790 & {5} & 945 & 1010 & {5.25} & 2584 & 145 & 235 & 571582 & 4.71 & 0 & 918 & 1059 & 12.4 & 3600 & 981 & 5947 & 434403 & 7.77 \\\hline
        \cellcolor[HTML]{D9E1F2}\{0.5, 0.7, 0.9\} & 773 & {5} & 929 & 1015 & {6.92} & 2714 & 218 & 342 & 579739 & 6.75 & 0 & 915 & 1083 & 14.5 & 3600 & 948 & 5051 & 489491 & 8.86 \\ \hline 
    \end{tabular}}
    \caption{Comparison between model~$(\mathsf{\textcolor{OliveGreen}{BPMD}})$ and $(\mathsf{\textcolor{OliveGreen}{BPMD\text{-}z}})$.}
    \label{tab:1}
\end{table}

\begin{table}[ht!]
    \centering
    \scalebox{0.6}{\begin{tabular}{|l|r|r r r r r r r r | r |r r r r r r r r | r |} \hline
        \multirow{2}{*}{} &  {Heuristic} &  \multicolumn{9}{c|}{Model~$(\mathsf{\textcolor{OliveGreen}{BPMD_d}})$} & \multicolumn{9}{c|}{Model~$(\mathsf{\textcolor{OliveGreen}{BPMD_d\text{-}z}})$} \\ \cline{2-20}
        & $\text{LB}_h$
        & \#opt & LB & UB & gap & time & septime & \#sep & \#nodes & $\overline{\text{gap}}$ &
        \#opt & LB & UB & gap & time & septime & \#sep & \#nodes & $\overline{\text{gap}}$   \\ \hline
        \textit{Chao instances} & \multicolumn{19}{c|}{} \\ \hline
        \cellcolor[HTML]{D9E1F2}\{0.2, 0.5\} & 1076  & 15 & 1076 & 1076 & 0.00 & 7.2 & 0.00 & 1 & 0 & 0.00 & 15 & 1076 & 1076 & 0.00 & 7.8 & 1.5 & 2 & 0 & 0.00  \\ \hline
        \cellcolor[HTML]{D9E1F2}\{0.5, 0.9\} & 1876  & 0 & 1889 & 1937 & 2.69 & 3600 & 1822 & 420 & 292460 & 2.69 & 0 & 1900 & 1937 & 1.91 & 3600 & 2655 & 5148 & 113127 & 1.91 \\ \hline
        \cellcolor[HTML]{D9E1F2}\{0.2, 0.5, 0.8\} & 1716 & 9 & 1718 & 1722 & 0.27 & 1446 & 893 & 176 & 87529 & 0.27 & 10 & 1718 & 1722 & {0.23} & 1320 & 1161 & 1544 & 25102 & {0.23} \\ \hline
        \cellcolor[HTML]{D9E1F2}\{0.5, 0.7, 0.9\} & 1881 & 0 & 1891 & 1937 & 2.50 & 3600 & 1606 & 336 & 243992 & 2.50 & 0 & 1908 & 1937 & 1.52 & 3600 & 2598 & 5286 & 146093 & 1.52 \\
        \hline\hline
        \textit{Solomon instances} & \multicolumn{19}{c|}{} \\ \hline
        \cellcolor[HTML]{D9E1F2}\{0.2, 0.5\} & 672  & 9 & 675 & 676 & 0.20 & 904 & 132 & 94 & 81940 & 0.20 & 9 & 675 & 676 & {0.10}  & 944 & 807 & 1733 & 15618 & 0.10 \\\hline
        \cellcolor[HTML]{D9E1F2}\{0.5, 0.9\} & 770 & 1 & 867 & 1015 & {13.3} & 3495 & 347 & 417 & 809036 & 12.2 & 0 & 842 & 1078 & 22.0 & 3600 & 1297 & 9379 & 305963 & 16.1 \\\hline
        \cellcolor[HTML]{D9E1F2}\{0.2, 0.5, 0.8\} & 790 & 1 & 929 & 1021 & {8.24} & 3405 & 300 & 418 & 655635 & 6.79 & 0 & 873 & 1072 & 18.5 & 3600 & 1211 & 8422 & 415383 & 12.8 \\\hline
        \cellcolor[HTML]{D9E1F2}\{0.5, 0.7, 0.9\} & 773 & 0 & 907 & 1031 & {11.0} & 3600 & 517 & 435 & 628041 & 9.5 & 0 & 884 & 1094 & 19.2 & 3600 & 1305 & 8771 & 366167 & 12.4 \\ \hline 
    \end{tabular}}
    \caption{Comparison between model~$(\mathsf{\textcolor{OliveGreen}{BPMD_d}})$ and $(\mathsf{\textcolor{OliveGreen}{BPMD_d\text{-}z}})$.}
    \label{tab:2}
\end{table}

\begin{table}[ht!]
\centering
    \scalebox{0.6}{\begin{tabular}{|l|r|r r r r r r r r | r |r r r r r r r r | r |} \hline
        \multirow{2}{*}{} &  {Heuristic} &  \multicolumn{9}{c|}{Model~$(\mathsf{\textcolor{OliveGreen}{BPMD_d}})$} & \multicolumn{9}{c|}{Model~$(\mathsf{\textcolor{OliveGreen}{BPMD_d\text{-}z}})$} \\ \cline{2-20}
        & $\text{LB}_h$
        & \#opt & LB & UB & gap & time & septime & \#sep & \#nodes & $\overline{\text{gap}}$ &
        \#opt & LB & UB & gap & time & septime & \#sep & \#nodes & $\overline{\text{gap}}$   \\ \hline
        \cellcolor[HTML]{D9E1F2}\{0.2, 0.5\} & 529 & {7} & 529 & 598 & {8.76} & 2402 & 305 & 34 & 615771 & 7.29 & 2 & 529 & 756 & 26.2 & 3141 & 2501 & 8893 & 61311 & 7.29 \\ \hline
        \cellcolor[HTML]{D9E1F2}\{0.5, 0.9\} & 941 & 6 & 954 & 1077 & {8.81} & 2465 & 141 & 44 & 530289 & 7.46 & 2 & 951 & 1367 & 27.0 & 3281 & 2181 & 11178 & 93651 & 7.89 \\ \hline 
        \cellcolor[HTML]{D9E1F2}\{0.2, 0.5, 0.8\} & 841 & 6 & 848 & 958 & {8.80} & 2480 & 143 & 43 & 477236 & 7.47 & 2 & 841 & 1222 & 28.2 &  3179 & 2062 & 9675 & 82686 & 8.14\\ \hline
        \cellcolor[HTML]{D9E1F2}\{0.5, 0.7, 0.9\} & 942 & 5 & 949 & 1089 & 9.85 & 2734 & 331 & 47 & 465270 & 7.96 & 3 & 948 & 1363 & 26.4 & 3216 & 2037 & 10116 & 70487 & 8.01 \\ \hline
    \end{tabular}}
    \caption{Comparison between model~$(\mathsf{\textcolor{OliveGreen}{BPMD}})$ and $(\mathsf{\textcolor{OliveGreen}{BPMD\text{-}z}})$ when considering the route duration limit.}
    \label{tab:3}
\end{table}
\begin{table}[ht!]
    \centering
    \scalebox{0.6}{\begin{tabular}{|l|r|r r r r r r r r | r |r r r r r r r r | r |} \hline
        \multirow{2}{*}{} &  {Heuristic} &  \multicolumn{9}{c|}{Model~$(\mathsf{\textcolor{OliveGreen}{BPMD_d}})$} & \multicolumn{9}{c|}{Model~$(\mathsf{\textcolor{OliveGreen}{BPMD_d\text{-}z}})$} \\ \cline{2-20}
        & $\text{LB}_h$
        & \#opt & LB & UB & gap & time & septime & \#sep & \#nodes & $\overline{\text{gap}}$ &
        \#opt & LB & UB & gap  & time & septime & \#sep & \#nodes & $\overline{\text{gap}}$  \\ \hline
        \cellcolor[HTML]{D9E1F2}\{0.2, 0.5\} & 529 & {7} & 529 & 612 & {10.3} & 2564 & 395 & 71 & 608464 & 7.26 & 0 & 529 & 805 & 34.0 & 3600 & 2206 & 12492 & 68683 & 7.29 \\ \hline 
        \cellcolor[HTML]{D9E1F2}\{0.5, 0.9\} & 941 & 6 & 949 & 1099 & 10.9 & 2760 & 406 & 74 & 722421 & 8.07 & 0 & 946 & 1449 & 34.5  & 3600 & 2279 & 12192 & 75795 & 8.71\\ \hline 
        \cellcolor[HTML]{D9E1F2}\{0.2, 0.5, 0.8\} & 839 & 5 & 849 & 991 & {11.8} & 3039 & 472 & 76 & 778847 & 7.57 & 0 & 844 & 1313 & 36.5 & 3600 & 2433 & 11663 & 62590 & 7.97 \\ \hline
        \cellcolor[HTML]{D9E1F2}\{0.5, 0.7, 0.9\} & 942 & 3 & 950 & 1129 & {14.4} & 3063 & 406 & 69 & 787943 & 8.10 & 0 & 946 & 1479 & 36.9 & 3600 & 2415 & 12466 & 79820 & 8.69 \\ \hline
    \end{tabular}}
    \caption{Comparison between model~$(\mathsf{\textcolor{OliveGreen}{BPMD_d}})$ and $(\mathsf{\textcolor{OliveGreen}{BPMD_d\text{-}z}})$ when considering the route duration limit.}
    \label{tab:4}
\end{table}

These tables clearly show that, if the high margin is \textit{not so high}, i.e., $\{0.2, 0.5\}$, then the leader can assign all (or almost all) the margins to high, i.e., $0.5$, and all the followers accept their assignment, thus the branching tree is very small, as the heuristic finds an optimal solution in most of the cases. Otherwise, more iterations will be performed, and more computing time is required. 
In Tables~\ref{tab:1}--\ref{tab:2}, we can observe two different trends for the \textit{Chao} and the \textit{Solomon instances}. For the \textit{Chao instances}, the gap at termination returned by the models without ${z}$ is lower. The opposite is true for the \textit{Solomon instances}. This is related to the fact that in the \textit{Solomon instances} the customers are more randomly distributed and further apart from each other, thus having information on the ${z}$ variables at the master level helps in the resolution. This information on the routing decisions is also useful for solving the \textit{Chao instances} when the limit on the route duration is imposed, as it is clear from Tables~\ref{tab:3}--\ref{tab:4}. 
Indeed, in models having ${z}$ variables in the upper level, constraint~\eqref{eq:pricing-tmax} can be directly imposed at the upper level and does not impact the separation procedure, as it happens instead in the case of constraint~\eqref{eq:proj-tmax}.

The values of $\overline{\text{gap}}$ are calculated ex-post, after taking into consideration both the upper bound coming from the branch-and-cut and the one coming from the solution of the $(\mathsf{\textcolor{OliveGreen}{WTA\text{-}PTP\text{-}MD}})$ formulation. Whereas the value of the gap can be used to compare the performances of the LP relaxations of the different formulations considered, the value of $\overline{\text{gap}}$, which is on average the smallest between the two, provides us with information regarding the quality of the best solution found by the different methods. The values of $\overline{\text{gap}}$ follow the same trends of the values of the classic gap for the \textit{Chao} and the \textit{Solomon instances}.

We can further notice that, in general, the number of branch-and-cut nodes required by the models without variables ${z}$ is smaller than the one required by the models with ${z}$, even if the percentage of computing time that is required for the separation procedure is higher. Indeed, for these models, at each node, we need to separate not only the value function cuts and the subtour elimination constraints~\eqref{eq:subtour}, but also constraints~\eqref{eq:routing-proj-theta}, which involve the solution of a TSP.

\begin{figure}[ht!]
    \centering
    \includegraphics[width=13cm,height=6.5cm]{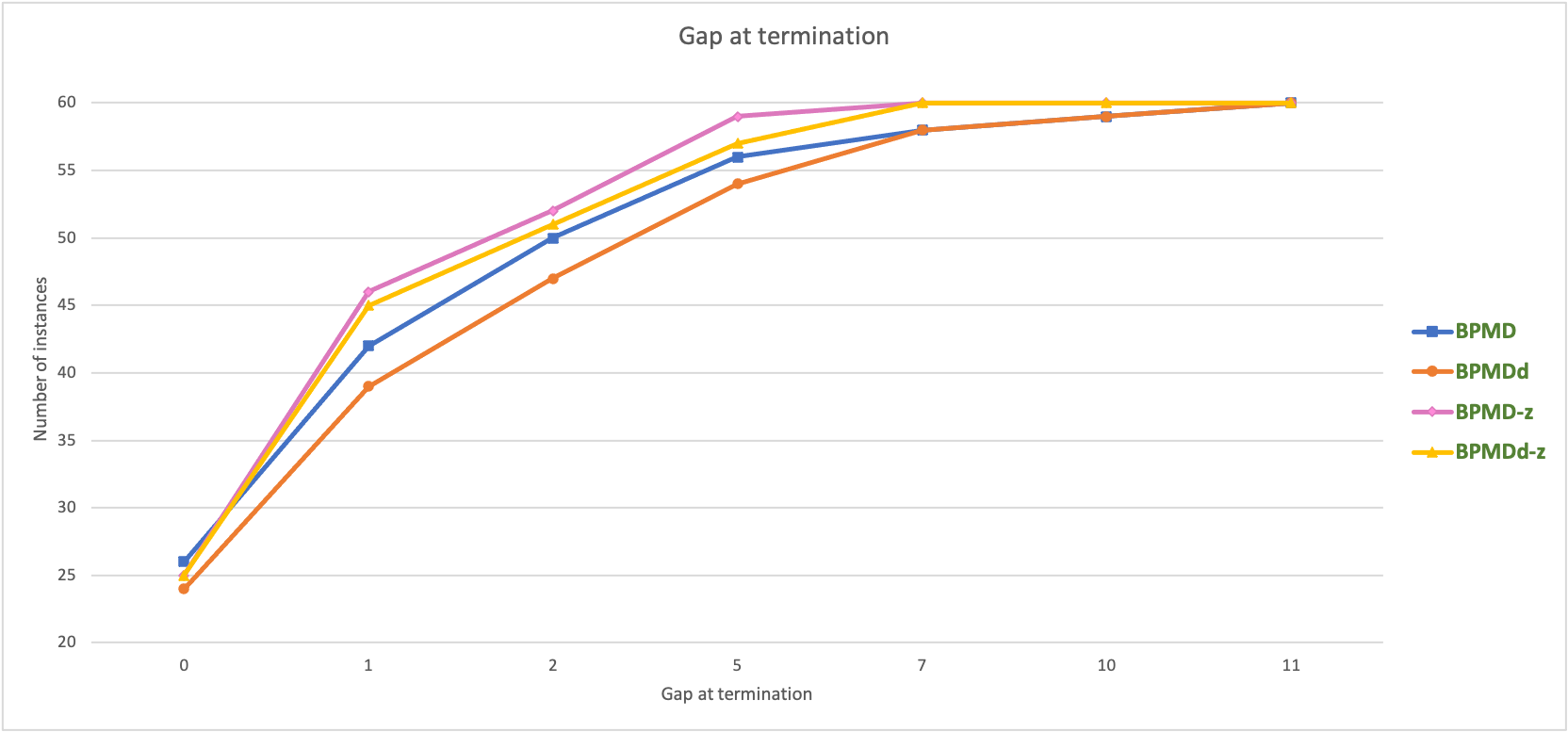}
    \caption{Cumulative chart of the number of \textit{Chao instances} solved within a given gap at termination.}
    \label{fig:chart1}
\end{figure}

\begin{figure}[ht!]
    \centering
    \includegraphics[width=13cm,height=6.5cm]{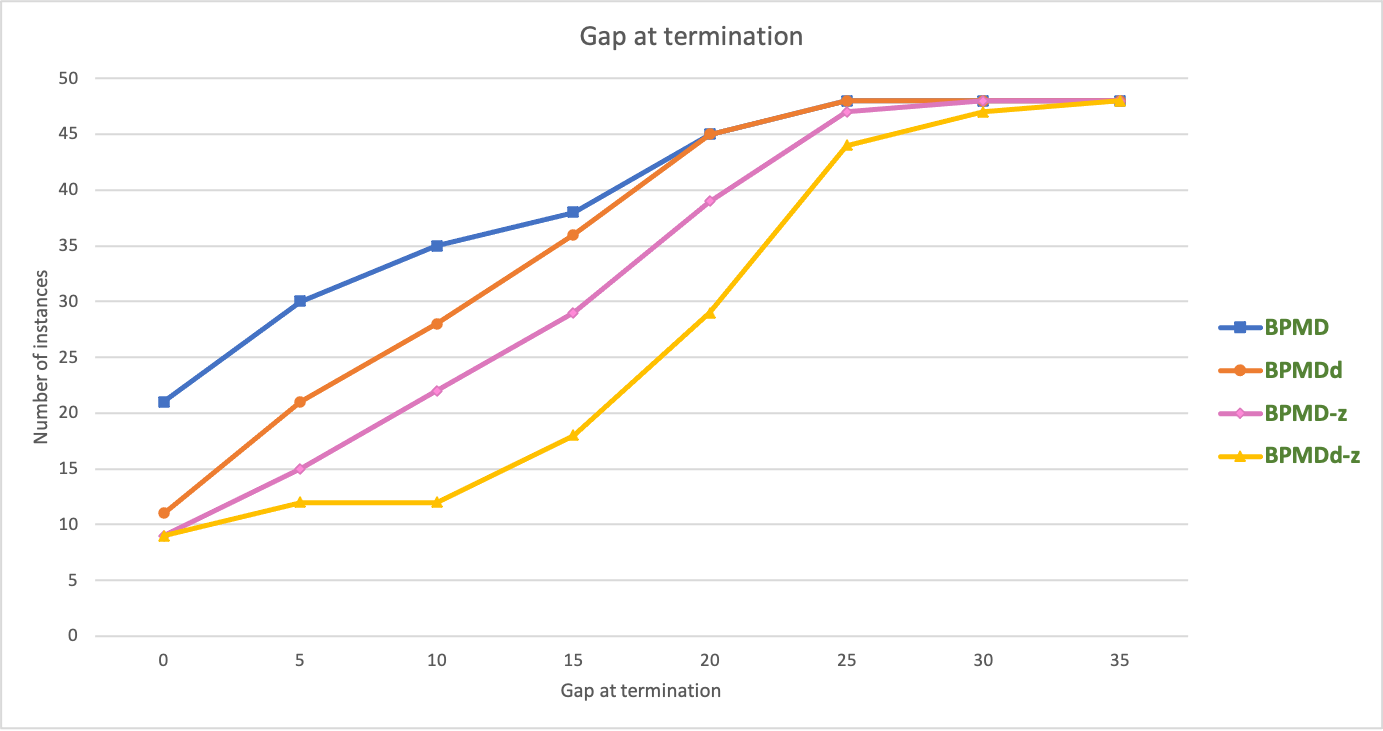}
    \caption{Cumulative chart of the number of \textit{Solomon instances} solved within a given gap at termination.}
    \label{fig:chart2}
\end{figure}

We further provide two summary charts in Figures~\ref{fig:chart1} and \ref{fig:chart2} related to the performance of the four models on the \textit{Chao instances} and \textit{Solomon instances}, respectively. They report the number of instances (on the vertical axis) for which the gap at termination is smaller than or equal to the value reported on the horizontal axis.
They confirm what is shown in the summary tables presented above. Indeed, in Figure~\ref{fig:chart1}, the curves associated with the models without ${z}$ are higher than the curves associated with the models with ${z}$; the opposite is true in Figure~\ref{fig:chart2}. In addition, we can notice that, disaggregated models~$(\mathsf{\textcolor{OliveGreen}{BPMD_d}})$ and $(\mathsf{\textcolor{OliveGreen}{BPMD_d\text{-}z}})$ are overall performing worse than models~$(\mathsf{\textcolor{OliveGreen}{BPMD}})$ and $(\mathsf{\textcolor{OliveGreen}{BPMD\text{-}z}})$, respectively. This might be related to the fact that the disaggregated model has more variables in the lower level. Indeed, on the one hand, these variables are also part of the single-level reformulation (due to the value-function approach), and on the other hand, they slow down the separation procedure, which involves the solution of the lower level.

The same conclusions can be drawn from the observations of the box plots reported in Figures~\ref{fig:box1a}-\ref{fig:box1b} and \ref{fig:box2a}-\ref{fig:box2b} for \textit{Chao instances} and \textit{Solomon instances}, respectively. These figures show the distribution of the gap at termination, grouped by the number of customers (Figures~\ref{fig:box1a} and \ref{fig:box2a}) or vehicles (Figures~\ref{fig:box1b} and \ref{fig:box2b}).

\begin{figure}[ht!]
    \centering
    \subfloat[Box plots obtained by aggregating the \textit{Chao instances} with the same number of customers.]{\includegraphics[width=13cm,height=6cm]{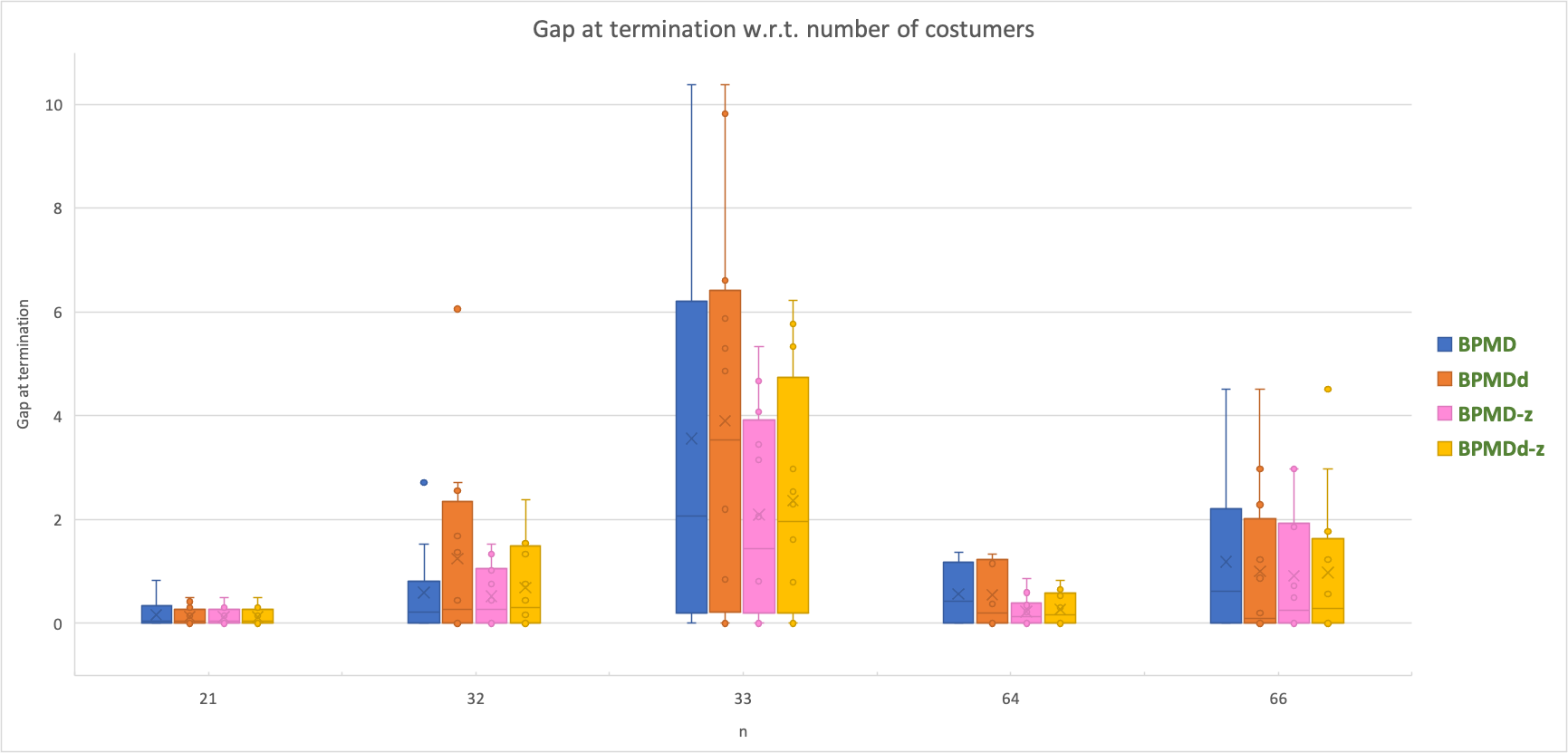}
        \label{fig:box1a}}\hfill
    \centering
    \subfloat[Box plots obtained by aggregating the \textit{Chao instances} with the same number of vehicles.]{\includegraphics[width=13cm,height=6cm]{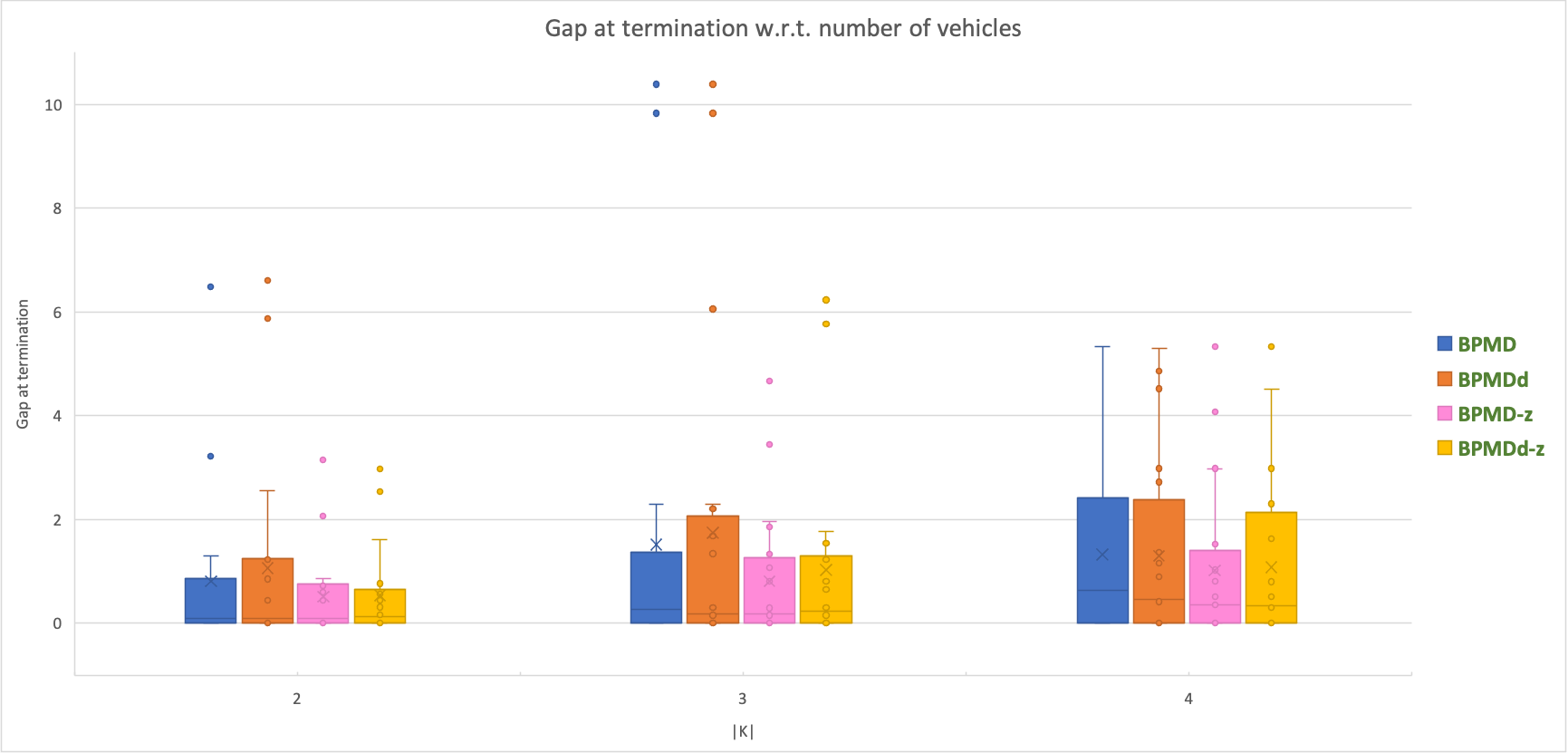}
        \label{fig:box1b}}
    \caption{Box plots representing the distribution of the gap at termination of \textit{Chao instances}.\vspace*{-3mm}}
\end{figure}
\begin{figure}[ht!]
    \centering
    \subfloat[Box plots obtained by aggregating the \textit{Solomon instances} with the same number of customers.]{\includegraphics[width=13cm,height=6cm]{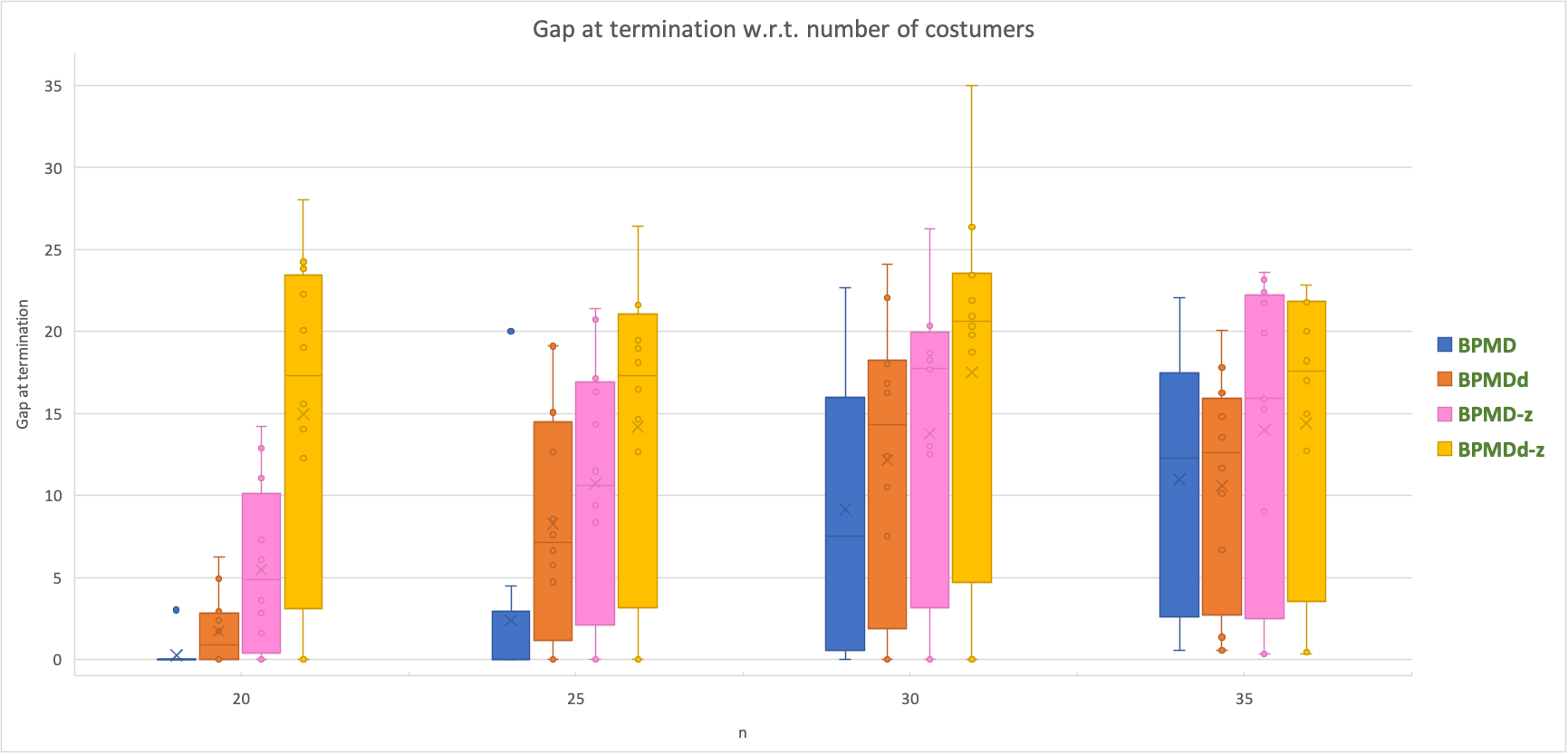}
        \label{fig:box2a}}\hfill
    \centering
    \subfloat[Box plots obtained by aggregating the \textit{Solomon instances} with the same number of vehicles.]{\includegraphics[width=13cm,height=6cm]{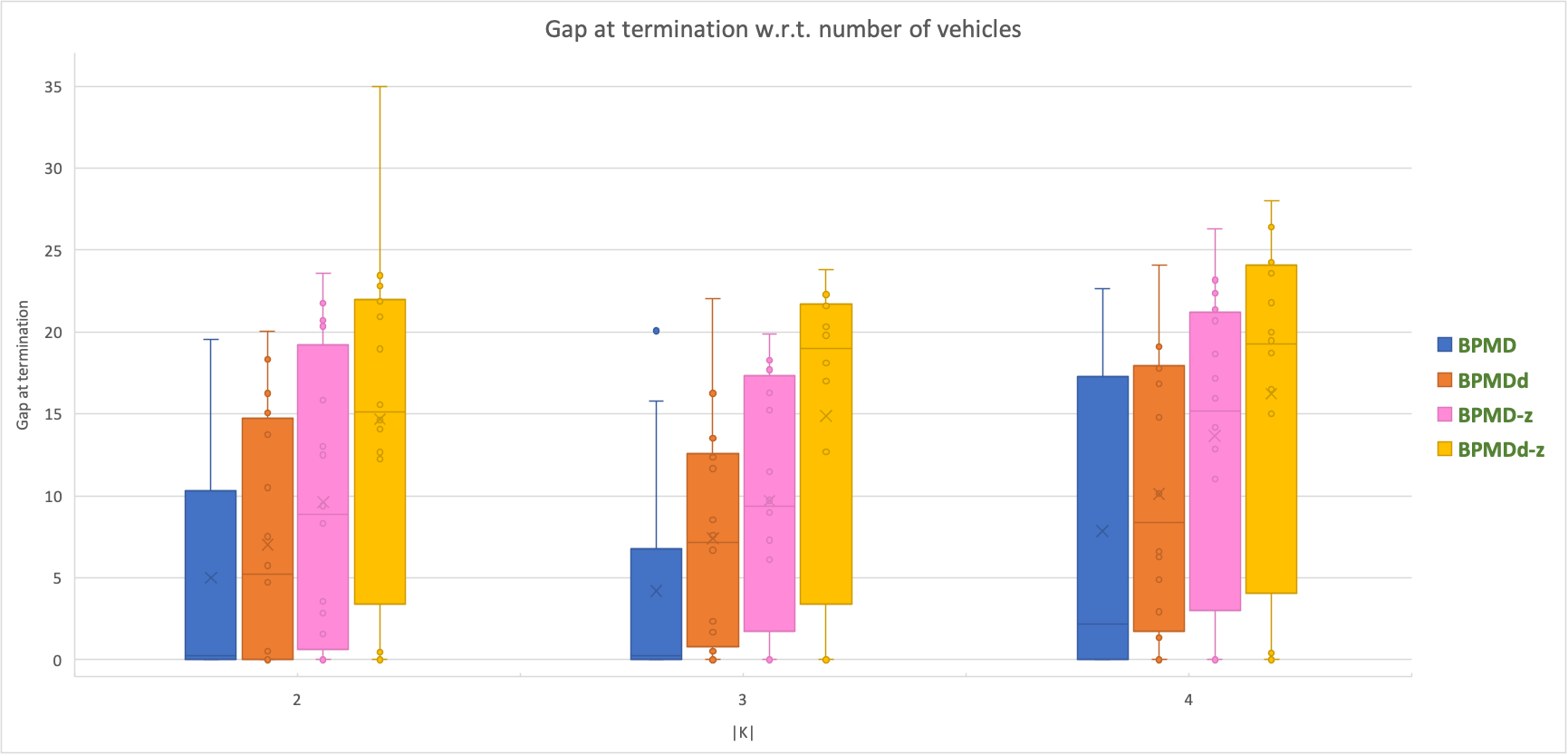}
        \label{fig:box2b}}
    \caption{Box plots representing the distribution of the gap at termination of \textit{Solomon instances}.\vspace*{-3mm}}
\end{figure}

\subsection{\texorpdfstring{The gain related to margin decisions}{The gain related to margin decisions}}\label{sub:gain}
To understand the structure of the solutions and the gain related to margin decisions, we choose two illustrative instances, namely the \textit{Solomon instances} R20\_2 and R20\_3, both of which could be solved to optimality by model~\eqref{eq:pricing1_single} for all the tested margin values, even without the warm-start solution provided by the heuristic.

Table~\ref{tab:5} has two blocks of columns: one is devoted to the results on the two considered instances when margin decisions are considered, and the other to the results when the margins are fixed to different values of $m \in M$. For the BPMD, for every tested set of margins $M$, we report the profit of the platform, the percentage of items served with high, medium, and low margins respectively, the percentage of served customers, and the computational time. For the BPFM, in every row we report the profit of the platform and the computational time when considering all margins fixed to either low, or high or, if $|M|=3$, medium, as well as when considering different random margins chosen within $M$ for every item.

From the first block of the table, it is evident that the higher the considered margins are, the more diverse the solutions are, in terms of the margins applied in optimal solutions, and the more difficult it is to solve them. Furthermore, for the leader, it seems more convenient to consider $i)$ higher margins, and $ii)$ a higher number of margins. When comparing the first and the second blocks of columns of the tables, we can clearly notice the added value in terms of the platform's profit from making margin decisions as compared to operating under fixed margins.

\begin{table}[ht!]
\centering
\scalebox{0.6}{\begin{tabular}{|l|c|rrr|r|r|rr|rr|rr|rr|}
    \hline
    & \multicolumn{6}{c}{\multirow{2}{*}{BPMD}} & \multicolumn{8}{|c|}{BPFM} \\ \cline{8-15}
    & \multicolumn{6}{c}{} & \multicolumn{2}{|c}{low} & \multicolumn{2}{|c}{medium} & \multicolumn{2}{|c}{high} & \multicolumn{2}{|c|}{random} \\ \hline
    & profit & \%high & \%medium & \%low & \%served & time & profit & time & profit & time & profit & time & profit & time \\ \hline
    R20\_2 & \multicolumn{14}{c|}{} \\ \hline
    \cellcolor[HTML]{D9E1F2}\{0.2, 0.5\} & 487.5 & 100.0 & - & 0.0 & 100.0 & 29.9 & 195.0 & 90.4 & - & - & 487.5 & 12.6 & 315.3 & 28.7 \\ 
    \cellcolor[HTML]{D9E1F2}\{0.5, 0.9\} & 675.7 & 41.2 & - & 58.8 & 89.5 & 805.2 & 487.5 & 12.6 & - & - & 0.0 & 0.23 & 527.9 & 980.7\\ 
    \cellcolor[HTML]{D9E1F2}\{0.2, 0.5, 0.8\} & 691.8 & 68.8 & 31.2 & 0.0 & 84.2 & 399.3 & 195.0 & 90.4 & 487.5 & 12.6 & 0.0 & 0.30 & 514.1 & 1462 \\ 
    \cellcolor[HTML]{D9E1F2}\{0.5, 0.7, 0.9\} & 731.8 & 33.3 & 44.4 & 22.2 & 90.0 & 500 & 487.5 & 12.6 & 692.3 & 59.6 & 0.0 & 54.9 & 612.5 & 436.1 \\ \hline
    R20\_3 & \multicolumn{14}{c|}{} \\ \hline
    \cellcolor[HTML]{D9E1F2}\{0.2, 0.5\} & 487.5 & 100.0 & - & 0.0 & 100.0 & 10.4 & 195.0 & 6.9 & - & - & 487.5 & 36.1 & 351.6 & 30.9 \\ 
    \cellcolor[HTML]{D9E1F2}\{0.5, 0.9\} & 661.3 & 41.2 & - & 58.8 & 89.5 & 1495 & 487.5 & 36.1 & - & - & 0.0 & 0.28 & 384 & 195.6\\ 
    \cellcolor[HTML]{D9E1F2}\{0.2, 0.5, 0.8\} & 675.0 & 62.5 & 37.5 & 0.0 & 84.2 & 1649 & 195.0 & 6.9 & 487.5 & 36.1 & 0.0 & 0.28 & 434.0 & 552.6 \\ 
    \cellcolor[HTML]{D9E1F2}\{0.5, 0.7, 0.9\} & 705.6 & 22.2 & 50.0 & 27.8 & 90.0 & 42.4 & 487.5 & 36.1 & 674.8 & 98.7 & 0.0 & 122.5 & 612.2 & 13.5 \\ \hline
\end{tabular}}
\caption{Structure of optimal solutions of models $(\mathsf{\textcolor{OliveGreen}{BPMD}})$ and $(\mathsf{\textcolor{OliveGreen}{BPFM}})$ on \textit{Solomon instances} R20\_2 and R20\_3.}
\label{tab:5}
\end{table}

\section{Conclusions}\label{sec:conc}

The last-mile delivery field is undergoing an unprecedented transformation in its operational procedures, primarily driven by the surge in e-commerce. This shift has had significant consequences on the way business is conducted. First, the volume of deliveries has increased substantially: as customers opt for online ordering over in-store purchases, their orders must be efficiently dispatched for delivery. Second, e-buyers are more and more demanding in terms of delivery times. Consequently, the demand for delivery services has become unpredictable and volatile, whereas opportunities for consolidation are reduced. To address these challenges, the companies in the field started developing new delivery strategies. One such strategy, which is gaining significant success, is related to peer-to-peer delivery. In this model, companies (referred to as platforms in this paper) receive delivery requests from customers and match them with independent carriers available to perform the deliveries. Unlike the traditional setting, carriers in peer-to-peer delivery do not work directly for the company, but they have their own objective, which might not always align with those of the company. Therefore, the challenge for the company lies in maximizing the profit from the delivery operations, taking into account carriers' objectives and behavior. 

In this paper, we study a bilevel compensation and routing problem arising in the context of peer-to-peer delivery. The problem combines the peer-to-peer logistic platform decisions about the assignment of items to the carriers and the compensation for each delivered item. The objective of the platform is to maximize the profit generated from the delivered items, all the while factoring in carriers' individual objectives when making assignment and compensation decisions. The bilevel nature of these problems is highlighted by presenting two single-level formulations that lead to either an overestimation or an underestimation of the platform's profit. After considering the fixed compensation setting, we propose two bilevel formulations for the compensation and routing problem, one with aggregated variables and the other with disaggregated variables. These formulations are then reformulated into single-level models, which are compared in terms of the quality of their linear relaxations. Additionally, we present equivalent formulations where routing variables are projected out. Computational tests show that the performance of the formulations depends on the features of the instances, i.e., compensation values and customers' geography. Whereas on average the disaggregated models are performing worse than the aggregated ones, projecting out the routing variables only helps for one type of instance. The numerical results confirm that solving the single-level formulations (with reduced or increased power of the carriers, respectively) leads to biases in the true platform solution values. Furthermore, the analysis of the structure of the solutions reveals that $i)$ including decisions on the margins results in better profits for the platform, $ii)$ the platform may benefit from offering higher compensations to carriers, resulting in a higher number of accepted offers.

Besides some natural extensions of the problem, such as considering multiple depots for the carriers, multiple vehicles in each subproblem, or introducing a penalty for the undelivered items, future research may explore the introduction of stochasticity into the problem setting, especially with regard to carriers' behavior. The challenge would be modeling the uncertainty, on one side, and adapting the methodologies proposed in this paper to deal with it, on the other side, or potentially, proposing ad-hoc modeling and approaches.
\medskip

\noindent\textbf{Acknowledgements:} The research of E.\ Fern\'andez has been partially supported through the Spanish Ministerio de Ciencia y Tecnolog\'ia and European Regional Development Funds (ERDF) through project MTM2019-105824GB-I00. 
The research of C.\ Archetti, M.\ Cerulli, and I.\ Ljubi\'c was partially funded by CY Initiative of Excellence, France (grant ``Investissements d’Avenir ANR-16-IDEX-0008''). This support is gratefully acknowledged.
\bibliography{Cerulli-Archetti-Fernandez-Ljubic2024}

\begin{thebibliography}{40}
\providecommand{\natexlab}[1]{#1}
\providecommand{\url}[1]{\texttt{#1}}
\expandafter\ifx\csname urlstyle\endcsname\relax
  \providecommand{\doi}[1]{doi: #1}\else
  \providecommand{\doi}{doi: \begingroup \urlstyle{rm}\Url}\fi

\bibitem[Agatz et~al.(2012)Agatz, Erera, Savelsbergh, and Wang]{agatz2012}
N.~Agatz, A.~Erera, M.~Savelsbergh, and X.~Wang.
\newblock Optimization for dynamic ride-sharing: A review.
\newblock \emph{European Journal of Operational Research}, 223\penalty0
  (2):\penalty0 295--303, 2012.
\newblock \doi{10.1016/j.ejor.2012.05.028}.

\bibitem[Alnaggar et~al.(2021)Alnaggar, Gzara, and
  Bookbinder]{ALNAGGAR2021102139}
A.~Alnaggar, F.~Gzara, and J.~H. Bookbinder.
\newblock Crowdsourced delivery: A review of platforms and academic literature.
\newblock \emph{Omega}, 98:\penalty0 102139, 2021.
\newblock \doi{10.1016/j.omega.2019.102139}.

\bibitem[Archetti and Bertazzi(2021)]{archetti2021}
C.~Archetti and L.~Bertazzi.
\newblock Recent challenges in routing and inventory routing: E-commerce and
  last-mile delivery.
\newblock \emph{Networks}, 77\penalty0 (2):\penalty0 255--268, 2021.
\newblock \doi{10.1002/net.21995}.

\bibitem[Arslan et~al.(2019)Arslan, Agatz, Kroon, and Zuidwijk]{arslan2019}
A.~M. Arslan, N.~Agatz, L.~Kroon, and R.~Zuidwijk.
\newblock Crowdsourced delivery—a dynamic pickup and delivery problem with ad
  hoc drivers.
\newblock \emph{Transportation Science}, 53\penalty0 (1):\penalty0 222--235,
  2019.
\newblock \doi{10.1287/trsc.2017.0803}.

\bibitem[Ausseil et~al.(2022)Ausseil, Pazour, and Ulmer]{ausseil2022}
R.~Ausseil, J.~A. Pazour, and M.~W. Ulmer.
\newblock Supplier menus for dynamic matching in peer-to-peer transportation
  platforms.
\newblock \emph{Transportation Science}, 56\penalty0 (5):\penalty0 1304--1326,
  2022.
\newblock \doi{10.1287/trsc.2022.1133}.

\bibitem[Barbosa et~al.(2023)Barbosa, Pedroso, and Viana]{barbosa2023}
M.~Barbosa, J.~P. Pedroso, and A.~Viana.
\newblock A data-driven compensation scheme for last-mile delivery with
  crowdsourcing.
\newblock \emph{Computers \& Operations Research}, 150:\penalty0 106059, 2023.
\newblock \doi{10.1016/j.cor.2022.106059}.

\bibitem[Calvete et~al.(2011)Calvete, Galé, and Oliveros]{calvete2011}
H.~Calvete, C.~Galé, and M.-J. Oliveros.
\newblock Bilevel model for production-distribution planning solved by using
  ant colony optimization.
\newblock \emph{Computers \& Operations Research}, 38:\penalty0 320--327, 01
  2011.
\newblock \doi{10.1016/j.cor.2010.05.007}.

\bibitem[Camacho-Vallejo et~al.(2021)Camacho-Vallejo, López-Vera, Smith, and
  González-Velarde]{camachovallejo2021}
J.-F. Camacho-Vallejo, L.~López-Vera, A.~Smith, and J.-L. González-Velarde.
\newblock A tabu search algorithm to solve a green logistics bi-objective
  bi-level problem.
\newblock \emph{Annals of Operations Research}, 07 2021.
\newblock \doi{10.1007/s10479-021-04195-w}.

\bibitem[Candler and Norton(1977)]{candler}
W.~Candler and R.~Norton.
\newblock Multi-level programming and development policy.
\newblock Technical report, The World Bank Development Research Center,
  Washington D.C., 1977.

\bibitem[Cerulli(2021)]{cerul2021}
M.~Cerulli.
\newblock \emph{Bilevel optimization and applications}.
\newblock PhD thesis, Institut Polytechnique de Paris, 2021.
\newblock URL \url{http://www.theses.fr/2021IPPAX108}.

\bibitem[Chao et~al.(1996)Chao, Golden, and Wasil]{chao1996}
I.-M. Chao, B.~L. Golden, and E.~A. Wasil.
\newblock The team orienteering problem.
\newblock \emph{European Journal of Operational Research}, 88\penalty0
  (3):\penalty0 464--474, 1996.
\newblock \doi{10.1016/0377-2217(94)00289-4}.

\bibitem[Cleophas et~al.(2019)Cleophas, Cottrill, Ehmke, and
  Tierney]{cleophas2019}
C.~Cleophas, C.~Cottrill, J.~F. Ehmke, and K.~Tierney.
\newblock Collaborative urban transportation: Recent advances in theory and
  practice.
\newblock \emph{European Journal of Operational Research}, 273\penalty0
  (3):\penalty0 801--816, 2019.
\newblock \doi{10.1016/j.ejor.2018.04.037}.

\bibitem[Colson et~al.(2007)Colson, Marcotte, and Savard]{colson}
B.~Colson, P.~Marcotte, and G.~Savard.
\newblock An overview of bilevel optimization.
\newblock \emph{Annals of Operations Research}, 153:\penalty0 235--256, 2007.
\newblock \doi{10.1007/s10479-007-0176-2}.

\bibitem[Dempe(2002)]{dempe}
S.~Dempe.
\newblock \emph{Foundations of bilevel programming}.
\newblock Springer Science \& Business Media, 2002.
\newblock \doi{10.1007/b101970}.

\bibitem[Du et~al.(2017)Du, Li, Yu, Dan, and Zhou]{du2017}
J.~Du, X.~Li, L.~Yu, R.~Dan, and J.~Zhou.
\newblock Multi-depot vehicle routing problem for hazardous materials
  transportation: A fuzzy bilevel programming.
\newblock \emph{Information Sciences}, 399:\penalty0 201--218, 2017.
\newblock \doi{10.1016/j.ins.2017.02.011}.

\bibitem[Feillet et~al.(2005)Feillet, Dejax, and Gendreau]{feillet2005}
D.~Feillet, P.~Dejax, and M.~Gendreau.
\newblock Traveling salesman problems with profits.
\newblock \emph{Transportation science}, 39\penalty0 (2):\penalty0 188--205,
  2005.
\newblock \doi{10.1287/trsc.1030.0079}.

\bibitem[Fischetti et~al.(1998)Fischetti, Gonz\'{a}lez, and
  Toth]{fischetti1998}
M.~Fischetti, J.~J.~S. Gonz\'{a}lez, and P.~Toth.
\newblock Solving the orienteering problem through branch-and-cut.
\newblock \emph{INFORMS Journal on Computing}, 10\penalty0 (2):\penalty0
  133--148, 1998.
\newblock \doi{10.1287/ijoc.10.2.133}.

\bibitem[Fortet(1960)]{fortet1960applications}
R.~Fortet.
\newblock Applications de l’algebre de boole en recherche op{\'e}rationelle.
\newblock \emph{Revue Fran{\c{c}}aise de Recherche Op{\'e}rationelle},
  4\penalty0 (14):\penalty0 17--26, 1960.

\bibitem[Gdowska et~al.(2018)Gdowska, Viana, and Pedroso]{gdowska2018}
K.~Gdowska, A.~Viana, and J.~P. Pedroso.
\newblock Stochastic last-mile delivery with crowdshipping.
\newblock \emph{Transportation Research Procedia}, 30:\penalty0 90--100, 2018.
\newblock \doi{10.1016/j.trpro.2018.09.011}.
\newblock EURO Mini Conference on ``Advances in Freight Transportation and
  Logistics''.

\bibitem[Handoko et~al.(2015)Handoko, Chuin, Gupta, Soon, Kim, and
  Siew]{handoko2015}
S.~D. Handoko, L.~H. Chuin, A.~Gupta, O.~Y. Soon, H.~C. Kim, and T.~P. Siew.
\newblock Solving multi-vehicle profitable tour problem via knowledge adoption
  in evolutionary bi-level programming.
\newblock In \emph{2015 IEEE Congress on Evolutionary Computation (CEC)}, pages
  2713--2720. IEEE, 2015.
\newblock \doi{10.1109/CEC.2015.7257225}.

\bibitem[Hong et~al.(2019)Hong, Li, He, Zhang, and Wang]{hong2019}
H.~Hong, X.~Li, D.~He, Y.~Zhang, and M.~Wang.
\newblock Crowdsourcing incentives for multi-hop urban parcel delivery network.
\newblock \emph{IEEE Access}, 7:\penalty0 26268--26277, 2019.
\newblock \doi{10.1109/ACCESS.2019.2896912}.

\bibitem[Horner et~al.(2021)Horner, Pazour, and Mitchell]{horner2021}
H.~Horner, J.~Pazour, and J.~E. Mitchell.
\newblock Optimizing driver menus under stochastic selection behavior for
  ridesharing and crowdsourced delivery.
\newblock \emph{Transportation Research Part E: Logistics and Transportation
  Review}, 153:\penalty0 102419, 2021.
\newblock \doi{10.1016/j.tre.2021.102419}.

\bibitem[IBM(2017)]{cplex}
IBM.
\newblock \emph{{ILOG CPLEX} 12.7 {\em User's Manual}}.
\newblock IBM, 2017.

\bibitem[Kafle et~al.(2017)Kafle, Zou, and Lin]{kafle2017}
N.~Kafle, B.~Zou, and J.~Lin.
\newblock Design and modeling of a crowdsource-enabled system for urban parcel
  relay and delivery.
\newblock \emph{Transportation Research Part B: Methodological}, 99:\penalty0
  62--82, 2017.
\newblock \doi{10.1016/j.trb.2016.12.022}.

\bibitem[Kleinert et~al.(2021)Kleinert, Labbé, Ljubić, and
  Schmidt]{ivana_survey}
T.~Kleinert, M.~Labbé, I.~Ljubić, and M.~Schmidt.
\newblock A survey on mixed-integer programming techniques in bilevel
  optimization.
\newblock \emph{EURO Journal on Computational Optimization}, 9:\penalty0
  100007, 2021.
\newblock \doi{10.1016/j.ejco.2021.100007}.

\bibitem[Le et~al.(2019)Le, Stathopoulos, {Van Woensel}, and Ukkusuri]{lei2019}
T.~V. Le, A.~Stathopoulos, T.~{Van Woensel}, and S.~V. Ukkusuri.
\newblock Supply, demand, operations, and management of crowd-shipping
  services: A review and empirical evidence.
\newblock \emph{Transportation Research Part C: Emerging Technologies},
  103:\penalty0 83--103, 2019.
\newblock \doi{10.1016/j.trc.2019.03.023}.

\bibitem[Marinakis and Marinaki(2008)]{marinakis2008}
Y.~Marinakis and M.~Marinaki.
\newblock A bilevel genetic algorithm for a real life location routing problem.
\newblock \emph{International Journal of Logistics Research and Applications},
  11\penalty0 (1):\penalty0 49--65, 2008.
\newblock \doi{10.1080/13675560701410144}.

\bibitem[Marinakis et~al.(2007)Marinakis, Migdalas, and Pardalos]{pardalos2007}
Y.~Marinakis, A.~Migdalas, and P.~Pardalos.
\newblock A new bilevel formulation for the {Vehicle Routing Problem} and a
  solution method using a genetic algorithm.
\newblock \emph{Journal of Global Optimization}, 38:\penalty0 555--580, 08
  2007.
\newblock \doi{10.1007/s10898-006-9094-0}.

\bibitem[Masoud and Jayakrishnan(2017)]{masoud2017}
N.~Masoud and R.~Jayakrishnan.
\newblock A decomposition algorithm to solve the multi-hop peer-to-peer
  ride-matching problem.
\newblock \emph{Transportation Research Part B: Methodological}, 99:\penalty0
  1--29, 2017.
\newblock \doi{10.1016/j.trb.2017.01.004}.

\bibitem[Mofidi and Pazour(2019)]{mofidi2019}
S.~S. Mofidi and J.~A. Pazour.
\newblock When is it beneficial to provide freelance suppliers with choice? a
  hierarchical approach for peer-to-peer logistics platforms.
\newblock \emph{Transportation Research Part B: Methodological}, 126:\penalty0
  1--23, 2019.
\newblock \doi{10.1016/j.trb.2019.05.008}.

\bibitem[Nikolakopoulos(2015)]{nikolakopoulos2015}
A.~Nikolakopoulos.
\newblock A metaheuristic reconstruction algorithm for solving bi-level vehicle
  routing problems with backhauls for army rapid fielding.
\newblock In V.~Zeimpekis, G.~Kaimakamis, and N.~Daras, editors, \emph{Military
  Logistics}, Operations Research/Computer Science Interfaces Series, pages
  141--157. Springer Cham, 2015.
\newblock \doi{10.1007/978-3-319-12075-1\_8}.

\bibitem[Ning and Su(2017)]{ning2017}
Y.~Ning and T.~Su.
\newblock {A multilevel approach for modelling vehicle routing problem with
  uncertain travelling time}.
\newblock \emph{Journal of Intelligent Manufacturing}, 28\penalty0
  (3):\penalty0 683--688, 2017.
\newblock \doi{10.1007/s10845-014-0979-3}.

\bibitem[Parvasi et~al.(2019)Parvasi, Tavakkoli-Moghaddam, Taleizadeh, and
  Soveizy]{parvasi2019}
S.~P. Parvasi, R.~Tavakkoli-Moghaddam, A.~Taleizadeh, and M.~Soveizy.
\newblock A bi-level bi-objective mathematical model for stop location in a
  school bus routing problem.
\newblock \emph{IFAC-PapersOnLine}, 52:\penalty0 1120--1125, 01 2019.
\newblock \doi{10.1016/j.ifacol.2019.11.346}.

\bibitem[Punel and Stathopoulos(2017)]{punel2017}
A.~Punel and A.~Stathopoulos.
\newblock {Modeling the acceptability of crowdsourced goods deliveries: Role of
  context and experience effects}.
\newblock \emph{Transportation Research Part E: Logistics and Transportation
  Review}, 105\penalty0 (C):\penalty0 18--38, 2017.
\newblock \doi{10.1016/j.tre.2017.06.007}.

\bibitem[Santos et~al.(2021)Santos, Curcio, Amorim, Carvalho, and
  Marques]{santos2020}
M.~J. Santos, E.~Curcio, P.~Amorim, M.~Carvalho, and A.~Marques.
\newblock A bilevel approach for the collaborative transportation planning
  problem.
\newblock \emph{International Journal of Production Economics}, 233, 2021.
\newblock \doi{10.1016/j.ijpe.2020.108004}.

\bibitem[Solomon(1987)]{solomon}
M.~M. Solomon.
\newblock Algorithms for the vehicle routing and scheduling problems with time
  window constraints.
\newblock \emph{Operations Research}, 35\penalty0 (2):\penalty0 254--265, 1987.
\newblock \doi{10.1287/opre.35.2.254}.

\bibitem[Tahernejad and Ralphs(2020)]{tahernejad2020valid}
S.~Tahernejad and T.~K. Ralphs.
\newblock Valid inequalities for mixed integer bilevel linear optimization
  problems.
\newblock Technical report, COR@L Technical Report 20T-013, 2020.

\bibitem[Wang et~al.(2021)Wang, Peng, and Xu]{wang2021}
C.~Wang, Z.~Peng, and X.~Xu.
\newblock A bi-level programming approach to the location-routing problem with
  cargo splitting under low-carbon policies.
\newblock \emph{Mathematics}, 9:\penalty0 2325, 2021.
\newblock \doi{10.3390/math9182325}.

\bibitem[Wang and Yang(2019)]{wang2019}
H.~Wang and H.~Yang.
\newblock Ridesourcing systems: A framework and review.
\newblock \emph{Transportation Research Part B: Methodological}, 129:\penalty0
  122--155, 2019.
\newblock \doi{10.1016/j.trb.2019.07.009}.

\bibitem[Wang et~al.(2018)Wang, Agatz, and Erera]{wang2018}
X.~Wang, N.~Agatz, and A.~Erera.
\newblock Stable matching for dynamic ride-sharing systems.
\newblock \emph{Transportation Science}, 52\penalty0 (4):\penalty0 850–867,
  2018.
\newblock \doi{10.1287/trsc.2017.0768}.

\end{thebibliography}

\end{document}